\documentclass[11pt,reqno]{amsart}

\usepackage{color}
\usepackage{amsmath, amsthm, amssymb}
\usepackage{amsfonts}
\usepackage[ansinew]{inputenc}
\usepackage[dvips]{epsfig}
\usepackage{graphicx}
\usepackage[english]{babel}
\usepackage{hyperref}

 \usepackage{mathrsfs}
\newcommand{\scrL }{\mathscr{L}}

\newcommand{\scrC }{\mathscr{C}}

\newcommand{\scrP }{\mathscr{P}}

\usepackage{cite}
\usepackage{graphicx}
\usepackage{amscd}
\usepackage{color}
\usepackage{bm}
\usepackage{enumerate}

\usepackage{verbatim}
\usepackage{hyperref}
\usepackage{amstext}
\usepackage{latexsym}
\theoremstyle{plain}
\newtheorem{Theorem}{Theorem}[section]

\newtheorem{Proposition}[Theorem]{Proposition}
\newtheorem{Lemma}[Theorem]{Lemma}

\newtheorem{Example}[Theorem]{Example}

\numberwithin{Theorem}{section}
\numberwithin{equation}{section}

\def\proof{\noindent{{\bf Proof. }}}
\def\square{\vbox{
\hrule height .4pt \hbox{\vrule width .4pt height 7pt \kern 7pt
\vrule width .4pt} \hrule height .4pt }}
\def\QED{\hfill {$\square$}\goodbreak \medskip}

\newcommand{\average}{{\mathchoice {\kern1ex\vcenter{\hrule height.4pt
width 6pt depth0pt} \kern-9.7pt} {\kern1ex\vcenter{\hrule
height.4pt width 4.3pt depth0pt} \kern-7pt} {} {} }}

\def\R{\mathbb{R}}

\def\div{\text{div}}

\renewcommand{\a }{\alpha }
\renewcommand{\b }{\beta }
\renewcommand{\d}{\delta }

\newcommand{\D }{\Delta }

\newcommand{\e }{\varepsilon }
\newcommand{\g }{\gamma}

\newcommand{\G }{\Gamma}
\renewcommand{\l }{\lambda }
\renewcommand{\L }{\Lambda }
\newcommand{\n }{\nabla }
\newcommand{\vp }{\varphi }

\newcommand{\s }{\sigma }

\newcommand{\z }{\zeta}
\renewcommand{\th }{\theta }

\renewcommand{\o }{\omega }
\renewcommand{\O }{\Omega }

\newcommand{\ov}{\overline}

\newcommand{\be}{\begin{equation}}
\newcommand{\ee}{\end{equation}}

\newcommand{\de}{\partial}

\newcommand{\ti}{\widetilde}

\renewcommand{\k}{\kappa}

\newcommand{\N}{\mathbb{N}}

\newcommand{\Z}{\mathbb{Z}}

\newcommand{\cA}{{\mathcal A}}
\newcommand{\cB}{{\mathcal B}}
\newcommand{\cC}{{\mathcal C}}
\newcommand{\cD}{{\mathcal D}}

\newcommand{\cF}{{\mathcal F}}
\newcommand{\cG}{{\mathcal G}}
\newcommand{\cH}{{\mathcal H}}

\newcommand{\cJ}{{\mathcal J}}
\newcommand{\cK}{{\mathcal K}}
\newcommand{\cL}{{\mathcal L}}
\newcommand{\cM}{{\mathcal M}}
\newcommand{\cN}{{\mathcal N}}
\newcommand{\cO}{{\mathcal O}}
\newcommand{\cP}{{\mathcal P}}
\newcommand{\cQ}{{\mathcal Q}}
\newcommand{\cR}{{\mathcal R}}
\newcommand{\cS}{{\mathcal S}}

\newcommand{\cU}{{\mathcal U}}

\renewcommand{\epsilon}{\varepsilon}

\begin{document}

\title[Periodic interface]
{Periodic   patterns for a model involving short-range and long-range interactions}

\author{Mouhamed Moustapha Fall}
\address{  African Institute for Mathematical Sciences in Senegal, KM 2, Route de
Joal, B.P. 14 18. Mbour, Senegal.}
\email{mouhamed.m.fall@aims-senegal.org}




\begin{abstract}
 We consider a physical  model where the total energy is governed by    surface tension  and attractive  screened Coulomb potential on the 3-dimensional space. We obtain different periodic equilibrium    patterns i.e. stationary sets for this energy, under some volume constraints.  The    patterns  bifurcate smoothly  from straight   lamellae,   lattices of round solid cylinders and  lattices of round balls.  
\end{abstract}

\maketitle
\section{Introduction and main results}
In this paper, we are interested with a class of periodic  stationary sets  in $\R^3$ for  the following energy functional
\be\label{eq:Geom-pblem-interface-like-s-perim}
\cP_\g(\O):=  |\de \O|+{\g }   \int_{\O}\int_{\O}{G}_{ \k}(|x-y|)  dxdy,
\ee
where $\k,\g>0$ and   $G_\k(r)=\frac{1}{r } e^{-\k r}$ is the repulsive Yukawa potential (or the \textit{screened} repulsive Coulomb potential).
The Yukawa potential $\ov{G}_\k(x)=G_\k(|x|)$ satisfies
\be 
-\D \ov{G}_\k+\k^2 \ov{G}_\k= 4\pi \d_0\qquad \textrm{ in $\R^3$.}
\ee
It  plays an  important role in the theory of elementary particles, see \cite{Yukawa}.  
For this energy  $\cP_\g$, comprising   short-range interaction   (surface tension) and  attractive Yukawa potential, a stationary set $\O$       may separate into well ordered disjoint components.  This   gives rise to formation of patterns    in the limit of    several nonlocal reaction-diffusion models from  physics, biology and chemistry. Indeed, systems with competing short range and   Coulomb-type  attractive interaction provides sharp interfaces to the   Ohta-Kawasaki model of block polymers and the activator-inhibitor  reaction-diffusion system, see for  for instance  \cite{GMP,SA, OS} and the  references therein. The paper of Muratov \cite{Muratov2002} contains a rigorous  derivation of  \eqref{eq:Geom-pblem-interface-like-s-perim} from the general   mean-field free energy functional, with $\g \approx \k^2$ and $\g,\k\to0$. Equilibrium patterns such as spots, stripes and the annuli, together with  morphlogical instabilities have been also studied in \cite{Muratov2002}. At  certain energy levels,    \eqref{eq:Geom-pblem-interface-like-s-perim} provides the    sharp interface for the Ohta-Kawasaki model of diblock copolymer melt, see Muratov   \cite{Muratov} and   the work of  Goldmann, Muratov and Serfaty  in \cite{Serfaty2013,Serfaty2014}  for  related variant of   $\Gamma$-convergence on the 2-dimensional torus.

 In \cite{Oshita}, Oshita showed that \eqref{eq:Geom-pblem-interface-like-s-perim} is the $\G$-limit for a class of activator-inhibitor  reaction-diffusion system, see also  Petrich and Goldstein in \cite{PD}, for the two dimensional case.    
 
 
   The question  of   minimizing   $\cP_\g$      has been also raised  by    H. Kn\"{u}pfer, C. Muratov and M. Novaga in \cite{KMN},  for   physical and mathematical perspectives. They speculated that bifurcations from trivial to nontirival solutions can be expected. Moreover they predict the  "pearl-necklace morphology"  (or string-of-pearl)   exhibited by long polyelectrolyte molecules in poor solvents, \cite{DR, FAM}.   In a geometric point of view, this is  alike the cylinder-to-sphere transition,  and known as Delauney surfaces (or unduloids). The  unduloids are  constant mean curvature surfaces of revolution interpolating between a straight cylinder to a string of  tangent round spheres.  Here, we will consider the limit configuration, and we prove existence of  stationary sets for $\cP_\g$ formed by a periodic  string of nearly round balls.\\

It is known, see for instance \cite{Muratov2002, FFMMM},  that  a stationary set (or critical point) $\O$ for the functional $\cP_\g$  satisfies the following Euler-Lagrange equation
\be \label{eq:main-problem}
H_{\de\O}(x)+\g \int_{\O}G_\k(|x-y|) dy=Const. \qquad \textrm{ for all } x\in \de\O,
\ee
where $H_{\de\O}$ is the  mean curvature (positive for a ball) of $\de\O$. Here and in the following,
for every set $\O$ with $C^2$ boundary (not necessarily bounded), we define the function 
$$
\cH_\O:\de\O\to \R
$$
given by
\be\label{eq:def-cH}
\cH_{\O}(x)=  H_{\de\O}(x)+\g \int_{\O}G_\k(|x-y|) dy,\qquad \textrm{ for every  $x\in \de\O$}.
\ee

In this paper, \textit{equilibrium   patterns} are sets $\O\subset\R^3$ for which   $\cH_\O\equiv Const$ on $\de \O$.\\

Our aim is to construct nontrivial    unbounded equilibrium patterns $\O$. These sets    are multiply periodic and bifurcate from      lattices of  round \textit{spheres},  straight  \textit{cylinders} and \textit{slabs}. In particular, we obtain stationary sets for $\cP_\g$ made by lattices of  near-spheres centered at any $M$-dimensional Bravais  lattice. When $M=1$, these are adjacent  nearly-spherical   sets centered at a straight line with equal gaps between them. The cylinders, result from the 2-dimensional existence of  spherical patterns with the corresponding Yukawa potential given by a modified Bessel function. These configurations appear for all $\g\in (0,\g_N)$, for some positive constant $\g_N$. Finally, we obtain lamellar  structures (modulated slabs) bifurcating from parallel planes.   The bifurcations of these modulated slabs   occur as long as $ \frac{\k\sqrt{\k^2+1}}{\sqrt{\k^2+1}-\k}<2\pi \g <(\k^2+1)^{3/2}$.

We first describe our  doubly periodic  slabs. We consider domains of the form,
\be \label{eq:domain-look-for}
\O_{\vp}=\left\{ (t,  z)\in
\mathbb{R}^{2}\times \R \,:\, -\vp(t )<z<\vp(t )  \right\} \subset\R^3,
\ee 
where  $\vp: \R^2 \to (0,\infty)$  is $2\pi\Z^2$-periodic and satisfies some symmetry properties. Indeed,  we let  $C^{k,\gamma}(\R^2)$ denotes the space of
$C^k(\R^2)$ bounded functions $u$, with bounded derivatives up to order $k$ and with $D^k u $ having finite H\"older 
seminorm of order $\a\in (0,1)$. The space $C^{k,\a}_p(\R^2) $ denotes the space of functions in $C^{k,\a}(\R^2)$ which are $2\pi\Z^2$-periodic.  Finally, we define
\begin{align} \label{def:cX-cP}
 C^{k,\a}_{p,\cR} 
 &=\left\{\vp \in C^{k,\a}_p(\R^2)\,: \, \textrm{   $\vp( t_1,t_2)=\vp(t_2,t_1)=\vp(-t_1,t_2)$, for all $  (t_1,t_2)\in   \R^2$} \right\}. \nonumber
\end{align}
The spaces $C^{k,\a}_{p,\cR}$ and $C^{k,\a}_{p}(\R^2)$  are equipped with the standard H\"older norm of $C^{k,\a} (\R^2) $. 
%

Our first main result is the following.

%
%


\begin{Theorem}\label{th:Slab}
Let  $\k>0$. Then for every   $\g$ satisfying  
\be\label{eq:bound-gamma-slab-int}
\frac{1}{2\pi} \frac{1}{  \frac{1}{\k}-  \frac{1}{\sqrt{\k^2+1 }}  }<  \g < \frac{1}{2\pi} \frac{   1}{  \frac{1}{(\k^2+1 )^{3/2}}  },
\ee
  there exists  $\l_*=\l_*(\k,\g)>0$  and a smooth curve 
$$
(-s_0,s_0)\to \R_+\times C^{2,\a}_{p,\cR},        \qquad s\mapsto (\l_s,  \vp_{s})
$$
with the   properties that $\l_0=\l_*$ and  $\O_{\vp_s}$ is an  equilibrium pattern     satisfying, for all $s\in (-s_0,s_0)$,    
$$
\cH_{\O_{\vp_{s}}} (x)=2\pi\g\,\textrm{arcsinh}(2\l_s/\k)   \qquad \textrm{  for every $x\in  \de\O_{\vp_{s}}$.}
$$
  Moreover,  for every $s \in (-s_0,s_0)$ and $(t_1,t_2)\in \R^2$, 
$$
\vp_{s}(t_1,t_2 )=\l_s+s\Bigl(  \cos(t_1)+\cos(t_2)+v_{s}(t_1,t_2 )\Bigr),
$$
with a smooth curve $ (-s_0,s_0) \to  C^{2,\alpha}_{p,\cR}(\mathbb{R}^2)$, $s\mapsto v_{s}$ satisfying
$$
\int_{[-\pi,\pi]^2} v_{s}(t) \cos( t_j) \,dt = 0 \qquad \text{for   $j=1,2$,}
$$
and $v_{0}\equiv 0$. 
 \end{Theorem}

The domains $\O_{\vp_{s}}$ in Theorem \ref{th:Slab}   bifurcates from the same  slab $\{-\l_*<z<\l_*$. For $\g$ satisfying $\frac{1}{2\pi} \frac{1}{  \frac{1}{\k}-  \frac{1}{\sqrt{\k^2+1 }}  }<  \g $ (see the lower bound  in \eqref{eq:bound-gamma-slab-int}), the constant $\l_*$ is the unique zero of the function
\be \label{eq:fonction-zero-lambda}
\l\mapsto  1 - 2\pi \g  \left\{  \frac{1}{\k}-  \frac{1}{\sqrt{\k^2+1  }} - \frac{\exp(-2\l\sqrt{\k^2+1}) }{ \sqrt{\k^2+1}} - \frac{\exp(-2\l\k) }{ \k}\right\}.
\ee
 The  Crandall-Rabonowitz bifurcation theorem  is the main tool we have used here to prove the existence of the branches $s\mapsto (\l_s,\vp_s)$. It provides a bifurcation result from  a \textit{  unique simple zero eigenvalue}.\\ 
We now discuss the condition \eqref{eq:bound-gamma-slab-int}. 
 If $ \g\leq  \frac{1}{2\pi} \frac{1}{  \frac{1}{\k}-  \frac{1}{\sqrt{\k^2+1 }}  } $ then the only solutions  bifurcating from $\O_{\l}$,  for $\l>0$, are the   trivial ones. This  follows from the fact that all eigenvalues of the lienearization of $\vp\mapsto \cH_{\O_\vp}$ at $\l$ are positive.  However,  equilibrium patterns with different morphology  (undulated cylinders), boundary of  Delauney-type surfaces,   exist    for $ \g< \frac{1}{2\pi} \frac{1}{  \frac{1}{\k}-  \frac{1}{\sqrt{\k^2+1 }}  } $, see   \cite{MMF}.  Now 
if $\g \geq  \frac{1}{2\pi} \frac{   1}{  \frac{1}{(\k^2+1 )^{3/2}}  }$, we do not know the number of zero eigenvalues of the linearized operator of $\vp\mapsto \cH_{\O_\vp}$ about the straight slab $\O_{\l_*}$, while the uniqueness of the trivial eigenvalue is crucial in our analysis. 

We describe next our (modulated) lamellae. 
For $\e\in \R_*$ close to zero, we   consider the   lamellae  made by the slices of  the  slabs  $\O_{\psi}$ above: 
\be \label{eq:def-Lamellae-intro}
\cL_{\psi}^\e=\O_{\psi}+\frac{1}{|\e|} \Z e_3=\bigcup_{p\in \Z}\left( \O_{\psi}+\left(0,0,\frac{p}{|\e|}\right)\right),
\ee
where $\O_\psi$ is as defined  in \eqref{eq:domain-look-for}.
It is easy to see that for  small $|\e|$, the  set $\cL_{\psi}^\e $ is as smooth as $\psi$. Moreover the straight  lamellae $\cL^\e_1$ is an  equilibrium pattern. More precisely,  $ \cH_{\cL^\e_\l}\equiv Const$ for every $\l>0$ and    $|\e|<\frac{1}{2\l}$.  We establish  the existence of nontrivial periodic patterns   bifurcating from  $\cL^\e_{\l_*} $, with $\l_*$ given by Theorem \ref{th:Slab}.

\begin{Theorem}\label{th:Lamellae}
Let $\k,\g,\l_s>0$, $\vp_s$ and  $s_0>0$  be  as in Theorem \ref{th:Slab}.  Then there exists $\e_0>0$ and a smooth branch 
$$(-\e_0,\e_0) \times   (-s_0,s_0)\to  \R_+\times  C^{2,\a}_{p,\cR}(\R^2), \qquad   \e\mapsto ( \l_{\e,s}, \psi_{\e,s})$$ such that  the set $\cL_{\psi_{\e,s}}^\e:= \O_{\psi_{\e,s}}+\frac{1}{|\e|} \Z e_3$ defined in \eqref{eq:def-Lamellae-intro} is an equilibrium pattern, with 
$$
\cH_{\cL_{\psi_{\e,s}}^\e } (x)=  \cH_{\cL_{\l_{\e,s}}^\e } \qquad \textrm{ for every $x\in \de \cL_{\psi_{\e,s}}^\e $},
$$
where, $\l_{0,s}=\l_s$,    $\psi_{0,s}=\vp_s$. Moreover, for every $s\in (s_0,s_0)$, $\e\in (-\e_0,\e_0)$ and   $(t_1,t_2)\in \R^2$, 
$$
\psi_{\e,s}(t_1,t_2)=   \l_{\e,s}+s \Bigl(   \cos(t_1)+\cos(t_2) +v_{\e, s}(t_1,t_2)\Bigr) 
$$
with     
$$
\int_{[-\pi,\pi]^2}v_{\e,s}(t)  \cos(t_j) dt=0\qquad \textrm{ for $j=1,2$.} 
$$
 
\end{Theorem}




As mentioned earlier, we are also interested on the cylindrical and spherical stationary  set for the energy $\cP_\g$.  These surfaces are obtained by perturbing balls in $\R^N$, with $N=2,3$,  centered at any $M$-dimensional  Bravais lattice for $1\leq M\leq N$. The perturbation parameter is a scaling parameter which makes the distance between the lattice points   sufficiently large. \\
Let $M\in\N$ with $1\leq M\leq 3$.  Let  $\left\{ \textbf{a}_1;\dots;\textbf{a}_M\right\}$  be a basis of $\R^M$.  
 We consider the    $M$-dimensional  {Bravais lattice} in $\R^3$:
$$
\scrL^M=\left\{\sum_{i=1}^M k_i \textbf{a}_i\,:\,  k=(k_1,\dots,k_M)\in \Z^M  \right\}\subset \R^3.
$$
Here and in the following, we identify $\R^M$ with $\R^M\times \{0_{\R^{3-M}}\}\subset \R^3$. Obviously  $\scrL^1=c\Z$, for some nonzero real number $c$. 

We define the perturbed cylinder
$$
C_\vp:=  \left\{( r,\th,t)\in
 \R_+\times S^1\times \mathbb{R} \,:\, r<\vp(\th)  \right\},
$$
where $\vp:S^1\to (0,\infty)$ is an even function.
For $\e>0$ small, we consider the lattice of perturbed cylinders centered at the lattice points  $\scrL^M$ given by   
\be \label{eq:def-latt-cyl-ontro}
  \cC_{\vp}^\e:=C_\vp+ \frac{1}{\e} \scrL^M.
\ee
We obtain the following result.
\begin{Theorem}\label{th:latt-cyl}
Let   $\scrL^M$  be an $M$-dimensional Bravais lattices, with $M=1,2$.
Then there exists $ \g_2>0$ such that for every $\g\in (0, \g_2)$, there exists     a smooth curve   $(0,\e_0)\to C^{2,\a}(S^1)$, $\e\mapsto \vp_\e$ such that for every $\e\in (0,\e_0)$,   the lattice of perturbed cylinders $  \cC_{\vp_\e}^\e$ given by   \eqref{eq:def-latt-cyl-ontro} is an equilibrium pattern, with 
$$
\cH_{  \cC_{\vp_\e}^\e}(x)= \cH_{  C_{1}}(\cdot)\qquad \textrm{    for every $x\in   \de  \cC_{\vp_\e}^\e$.}
$$
Moreover $\vp_\e=1+\o_\e$ with  $\o_\e$ even and $\|\o_\e\|_{C^{2,\a}(S^1) }\leq e^{-\frac{c}{\e}}$ as $\e\to 0$, for some positive constant $c$ depending only on $\k,\a,\g$ and $\scrL^M$. When $M=1$, the function $\o_\e$ is not constant. 
\end{Theorem}

  We emphasize that Theorem \ref{th:latt-cyl} results from a two-dimensional study. Indeed spherical equilibrium patterns  in 2-dimension with the corresponding 2-dimensional screened Coulomb  potential  (given by the second of  modified Bessel function $2 K_0(\k r)$) yields cylinderical patterns in 3-dimension with the Yukawa potential $G_\k(r)$.   
   In  \cite{CO},    Chen and Oshita  showed that among all 2-dimensional  Bravais lattices of  near-balls, 
the  hexagonal one provides least possible energy. \\
We expect that the perturbation $\o_\e$ is not constant on $S^1$. We are   able to prove this fact only in the case of lower dimensional lattice, $M=1$, see Section \ref{s:Remarks} below. The fact that the perfect 2-dimensional  lattice of cylinders $  \cC_{\vp}^\e:=C_1+ \frac{1}{\e} \scrL^2$ is not an equilibrium pattern  remains an open question. However,   numerical simulations give evidence that they are not constant, at  least in the hexagonal configuration, see Muratov \cite{Muratov2002}, for $0<\g=O(\k^2)\ll1$.    \\

To introduce the sphere lattices, we consider  the  $M$-dimensional Bravais $\scrL^M$ as above, with $M\leq 3$.   
Here,  $ S^{2}$ is  the $2$-dimensional unit sphere.
We   consider  the perturbed ball  
$$
B_\phi^3:= \{(r, \th )\in \R_+\times S^{2}\,:\, 0<r<\phi(\th )\},
$$
where $\phi: S^{2}\to (0,\infty)$ is an even function. 
 For $\e$ small enough, we consider the following disjoint union of perturbed balls
\be \label{eq:def-cB-phi-intro}
\cB_\phi^{\e} := B_{\phi}^3+\frac{1}{\e} \scrL^M.
\ee
We  have the following result.
\begin{Theorem}\label{th:latt-sph}
There exists $ \g_3>0$ such that for every $\g\in (0, \g_3)$, there exists   a smooth curve   $(0,\e_0)\to C^{2,\a}(S^{2})$, $\e\mapsto \phi_\e$   such that for every $\e\in (0,\e_0)$,    the set $\cB_{\phi_\e}^\e$ defined by  \eqref{eq:def-cB-phi-intro}  is an equilibrium pattern, with 
$$ 
\cH_{ \cB_{\phi_\e}^{\e}} (x)= \cH_{  B^3_1} (\cdot)\qquad \textrm{  for every $x\in    \de  \cB_{\phi_\e}^{\e}$.}
$$
Moreover ${\phi_\e}=1+w_\e$ with  $\o_\e$ even and $\|w_\e\|_{C^{2,\a}(S^{2}) }\leq e^{-\frac{c}{\e}}$ as $\e\to 0$, for some positive constant $c$ depending only on $\k,\a,\g$ and $\scrL^M$. When $M=2$, the function $w_\e$ is not constant. 
\end{Theorem}
The  constants $\g_N$, for $N=2,3,$ in Theorem \ref{th:latt-sph} and  Theorem \ref{th:latt-cyl} depend only on the eigenvalues of the linearized operator of the map  $w\mapsto \cH_{B^{N}_w}$, see    Lemma \ref{eq:Lambda-star-latt-sph} below.

In the case of near-sphere lattices also, we are   able to prove that $w_\e$ is not constant on $S^{N-1}$ only in the lower dimensional lattices, $M\leq 2$, see Section \ref{s:Remarks} below. It is an open question to know whether for   $M=3$, the function $w_\e$ is not constant i.e. $\cH_{\cB^\e_1}$ is not constant on $\de B_1$.
%

%
  Related to results in  the present  paper are  the  works, for  $\k=0$,    of  Ren and Wei    in their series of papers \cite{RW-1,RW-2,RW-3,RW-4,RW-5,RW-6,RW-7,RW-8,RW-9,RW-10,RW-11}. Among many others, they    build equilibrium patterns of type lamellar,
ring/cylindrical, and spot/spherical solutions.  See also \cite{KR-1, KR-2}   and the references therein. \\
Next, we emphasize some   similarities between  equilibrium patterns  derived here and  the structures of Constant Nonlocal Mean Curvatures (CNMC) sets recently discovered. The notion of nonlocal mean curvature was recently introduced by Caffarelli and Souganidis in \cite{Caff-Soug2010}. See Caffarelli, Roquejoffre, and Savin~\cite{Caffarelli2010},  it  arises from the  Euler-Lagrange equation for the fractional perimeter 
functional which in turn   is a  $\Gamma$-limit of some nonlocal phase transition problems, see \cite{SO}.\\
The long range interaction in the expression of $\cH_\O$  is, up to an additional constant,    a weighted average of all   "+1" in $\O$ and all   "-1" from outside $\O$,  since, 
\be\label{eq:def-cH-with -tau}
\cH_\O(x)=   H_{\de\O}(x)+\frac{\g}{2} \int_{\R^3}(1_{\O}(y)-1_{\R^3\setminus\O}(y))G_\k(|x-y|) dy+ \ov{c},
\ee
for $x\in \de\O$,
where $\ov{c}=\frac{1}{2}\int_{\R^3}G_\k(|x|)dx=\frac{2\pi}{\k^2}$. Similarly, for $\a\in (0,1)$,  the  fractional or nonlocal mean curvature function $\cH_\O^\a:\de\O\to \R$ is given by
$$
\cH_\O^\a(x):=- \frac{1-\a}{|B^2|} \int_{\R^3}(1_{\O}(y)-1_{\R^3\setminus\O}(y))|x-y|^{-3-\a} dy
$$
(positive if $\O$ is a ball).  
 The factor  $ \frac{1-\a}{|B^2|}$ implies that   $H_{\de \O}^\a(x)$ converges to the  mean curvature $ H_{\de \O}(x)$  as $\a\to1$, see    \cite{Davila2014B,Abatangelo}. 

 We can parallel the  recently constructed  CNMC sets  and  the equilibrium patterns    obtained here.
 Indeed,    by  \cite{Cabre2015A, CFW},  CNMC bands  exist, and from which 1-periodic CNMC slabs can be derived. Modulated CNMC slabs are found in \cite{MT}.  Finally from   \cite{Cabre2015B}, there exist     CNMC near-sphere lattices  from which CNMC lattices of  near-cylinders can be derived.  \\
Finally, we  point out that  resemblance between  phase structures in diblock copolymer melts,
and  triply periodic constant mean curvature  surfaces was also observed several years ago, see e.g. \cite{BF,TAHH}.\\

\section{Preliminary and notations}

We let $K_\nu$ be the second kind of  modified Bessel function of  order $\nu \in \R$. See e.g.
\cite{El},  for $\nu\in \R$,  there holds
 $$
K_\nu'(r)=-\frac{\nu}{r}K_\nu(r) -K_{\nu-1}(r)
$$
and $K_{-\nu}=K_{\nu}$ for $\nu<0$. 
Therefore 
\be \label{eq:deriv-K10}
K_0'(r)=-K_1(r), \qquad (rK_1)'(r)= -r K_{0}(r).
\ee     
See also
\cite{El},  for $\nu\geq 0$,
\be\label{eq:decKnuInf}
K_{\nu}(r)\sim \frac{\sqrt{\pi }}{\sqrt{2}} r^{-1/2}e^{-r} \qquad \textrm{  as
$r\to +\infty$} \ee 
and
\be\label{eq:decKnuInf0}
K_{0}(r)\sim -\log (r),  \qquad K_{\nu}(r)\sim \frac{1}{2}\G(\nu)   (r/2)^{-\nu} \qquad\textrm{  as
$r\to 0$.}  
\ee

 We now collect some useful   formula, where most of them can be found in \cite{GR}. In the following, we let $\a,\b,\d\geq 0$ and $\k>0$. 
By \cite[page 491, ET I 16(27)]{GR}, we have 
\be\label{eq:int-cos-G-k-sqrt}
 \int_{\R} \cos(\b  r) G_\k\left(\sqrt{r^2+\a ^2} \right)dr = 2  K_0(\a \sqrt{\k^2+\b^2}).
\ee
See also  \cite[page 719, WA 425(10)a, MO 48]{GR}, we have 
\be\label{eq:int-cos-K0}
\int_{\R}\cos(\b r) K_0(r\sqrt{\k^2+\a^2})dr= \frac{\pi}{\sqrt{\k^2+\b^2+\a^2  }}.
\ee
By computing, we have 
\be\label{eq:int-R2-G-k-sqrt}
 \int_{\R^2}   G_\k\left(|x|\right)dx=\frac{2\pi}{\k}.
\ee
From \cite[page 722, ET I 56(43)]{GR}, we get
 \be\label{eq:int-cos-K0-sqrt}
\int_{\R} \cos(\b r) K_0\left( \sqrt{r^2+\d^2}\sqrt{\a^2+\k^2} \right)  dr =   \pi\frac{\exp(-\d \sqrt{\b^2+\a^2 +\k^2})}{\sqrt{\b^2+\a^2 +\k^2}}.
 \ee
 Combining \eqref{eq:int-cos-G-k-sqrt} and  \eqref{eq:int-cos-K0-sqrt}, we obtain 
 \be\label{eq:int-G-k-R2-sqrt}
  \int_{\R^2}  G_\k\left(\sqrt{|x|^2+\d^2 }\right) dx =  2 \pi\frac{\exp(-\d \k) }{ \k}. 
 \ee

%
For a finite set $\cN$, we let $|\cN|$ denote the length (cardinal) of $\cN$. It will be understood that $ |\emptyset|=0$. 
Let $Z$ be a  Banach space and $U$ a nonempty open subset of $Z$. If $T \in C^{k}(U,\R)$ and $u \in U$, then $D^kT(u)$ is a continuous 
symmetric $k$-linear form on $Z$ whose norm is given by  
$$
   \|D^{k}T ({u}) \|= \sup_{{u}_{1}, \dots,  {u}_{k}\in Z }
     \frac{|D^{k} T ({u})[u_1,\dots,u_k]| }{   \prod_{j=1}^k \|   {u}_{j} \|_{ Z }}    .
$$
 If,  $L: Z\to \R$ is a linear map, we have 
\be \label{eq:Dk-LT2}
D^{|\cN|}(L T_2 )({u})[u_i]_{i\in\cN}= L({u})    D^{|\cN|} T_2({u})[u_i]_{i\in\cN} +  
\sum_{j\in\cN} L({u}_j)   D^{|\cN|-1} T_2({u})  [u_i]_{\stackrel{i \in \cN}{ i\neq j}}.
\ee
We let $T$ be as above, $V \subset \R$ open with $T(U) \subset V$ and $g:V  \to \R$ be a  $k$-times differentiable map.  The Fa\'{a} de Bruno formula states that 
\be 
\label{eq:Faa-de-Bruno}
D^k( g\circ T)(u)[u_1,\dots,u_k]= \sum_{\Pi\in\scrP_k} g^{ (\left|\Pi\right|)}(T(u)) \prod_{P\in\Pi} D^{\left|P\right| }T(u)[u_j]_{j \in P} ,
 \ee
for $u, u_1,\dots,u_k  \in U$, where $\scrP_k$ denotes the set of all partitions of  $\left\{ 1,\dots, k \right\}$, see e.g. \cite{FaadeBruno-JW}. 
 
\section{  Perodic slabs}\label{s:const-slab}

%
\subsection{Doubly periodic slabs}\label{ss:double-slabs}
In this section we construct equilibrium patterns  with period cell   given by a doubly periodic  slab. That is  domains of the form
$$
\O_\vp:=\left\{ (t, z)\in
\mathbb{R}^{2}\times \R \,:\, -\vp(t)<z<\vp(t)  \right\} \subset\R^3,
$$
where  $\vp: \R^2 \to (0,\infty)$, with $\vp\in C^{2,\a}(\R^2)$. Recall that   for $x\in \{-1,1\} $, the mean curvature is given by
$$
H_{\de\O_\vp}((t, \vp(t)x) )= -\div  \frac{\n \vp(t)}{\sqrt{1+|\n \vp(t)|^2}} .
$$
Moreover,  by a change of variable, the Yukawa interaction takes the form
\begin{align*}
&\int_{\O_\vp }G_\k(|(t, \vp(t)x)-X|) dX=\int_{\R^2} \int_{-\vp(s)}^{\vp(s)}   G_\k(|(t, \vp(t)x)-(s,z)|)dz ds\\
&= \int_{\R^2} \int_{-1}^1\vp(s) G_\k(|(t, \vp(t)x)-(s,\vp(s)y)|) dy ds\\
&=  \int_{\R^2} \int_{-1}^1\vp(s) G_\k(|(t, \vp(t))-(s,\vp(s)y)|) dy ds.
\end{align*}
%
%
Next, we define  the open set 
$$
\cO:=\{ \vp \in C^{2,\alpha}(\R^2)\,:\, \vp>0 \}
$$
and  the map $\cF: \cO\to C^{0,\a}(\R^2)$ by 
\begin{align}
&\cF(\vp)(t):= -\div \frac{\n \vp(t)}{\sqrt{1+|\n \vp(t)|^2}}   +\g  \int_{\R^2}\int_{-1 }^1\vp(s)G_\k(|(t, \vp(t) )-(s,\vp(s)y)|) dyds,\label{eq:dexp-cF-slab}
\end{align}
so that, for every $x\in \{-1,1\}$  and $t\in \R^2$, 
$$
\cH_{\O_\vp}((t,\vp(t)x))= \cF(\vp)(t).
$$
Our aim is to apply the Crandall-Rabinowitz bifurcation theorem to solve the equation
$$
\cF(\l+ u)=const
$$
for  $\l >0$ and $u\in C^{2,\a}(\R^2)$ is small. We will achieve this by solving the equation 
\be \label{eq:equation-to solve-slab}
\cF(\l+ u)=\cF(\l).
\ee
We note that by some computation and  \eqref{eq:int-G-k-R2-sqrt},    we have 
\be\label{eq:cF-lambda-eval}
 \cF(\l)= \g \l \int_{-1}^1\int_{\R^2}G_\k(\sqrt{|r|^2+\l^2(1-y)^2})dr dy=2\pi\g\,\textrm{arcsinh}(2\l/\k).
\ee
Let us study the linearized operator $D\cF(\l)$, for $\l>0$.
 We consider the map 
$$
\cO  \to  C^{0,\a}(\R^2),\qquad   \cF_1(\vp)(t):= \int_{\R^2}\vp(s) \int_{-1 }^1G_\k((|t-s|^2+(\vp(t) -\vp(s)y)^2)^{1/2}) dyds .
$$
We have the following result.
\begin{Lemma}\label{lem:Diff-of-cF-slab}
The map $\cF: \cO\to C^{0,\a}(\R^2)$ is smooth, with 
 $$
\|D^k \cF(\vp)\|\leq c (1+\|\vp\|_{C^{2,\a}(\R^2)})^c ,
$$
for some positive constant  $c=c(\k,\a,\d,k)$.
 Moreover for every $\vp\in \cO$  and $w\in C^{2,\a}(\R^2)$,
\begin{align}
\label{eq:DefcF1-slab-General}
D\cF_1(\vp)[w](t) 
&=  \int_{\R^2}  (w(t-r)-w(t))   G_\k((|r|^2+ ( \vp(t)  - \vp(t-r))^2)^{1/2})  dr\nonumber\\
&\quad+\int_{\R^2}  (w(t-r)  +w(t))G_\k((|r|^2+ ( \vp(t)  + \vp(t-r))^2)^{1/2}) dr .
\end{align}
In  particular, for every $\l>0$,
\begin{align} \label{eq:DcF-lamb-slab}
&D\cF(\l)[w](t) =-\D w(t)  - \g   \int_{\R^2} (w(t)-w(t-r)) G_\k\left(|r| \right)dr\nonumber \\
  &\quad +\g (  w(t)+w(t-r)) \int_{\R^2}  G_\k\left((|r|^2+4 \l^2 )^{1/2} \right) dr       .
\end{align}
\end{Lemma}
\begin{proof}
We notice that 
\begin{align}
&\cF(\vp)(t) =-\div \frac{\n \vp(t)}{\sqrt{1+|\n \vp(t)|^2}} \nonumber\\
&\quad  +\g  \int_{\R^2}\vp(s) \int_{-1 }^1G_\k((|t-s|^2+(\vp(t) -\vp(s)y)^2)^{1/2}) dyds.
\end{align}
The regularity of the map
$$
C^{2,\a}(\R^2)\to C^{0,\a}(\R^2),\qquad  \vp\mapsto  \div \frac{\n \vp(t)}{\sqrt{1+|\n \vp(t)|^2}} 
$$
is easy, we therefore consider the second term.  Fro every $\d>0$, we define 
$$
\cO_\d=\{ \vp\in C^{2,\a}(\R^2)\:\, \vp>\d\}.
$$
Let us first compute the formal derivative of this map which we will show that it is smooth. Once this is done we can easily deduce the regularity of $\cF_1$ and hence of $\cF$. We have 
\begin{align}\label{eq:DefcF1-slab-proof}
D\cF_1(\vp)[w](t)=\int_{\R^2}w(s) \int_{-1}^1f(y)dy ds+  \int_{\R^2} \int_{-1}^1 (w(s)y-w(t))  f'(y)dy ds,
\end{align}
where 
$$
f(y)=G_\k(|(t, \vp(t)  )-(s,\vp(s)y)|) 
$$
which satisfy 
$$
  f'(y)=\vp(s)(\vp(s)y- \vp(t)  ) G_\k'(|(t, \vp(t) )-(s,\vp(s)y)|) .
$$
Integration by parts give
\begin{align*}
\int_{-1}^1& (w(s)y-w(t) )   f'(y)dy=-\int_{-1}^1  w(s) f(y)dy+  (w(s)\th-w(t) )   f(\th)\Big|^{\th=1}_{\th=-1}.
\end{align*}
Inserting this in \eqref{eq:DefcF1-slab-proof}, we deduce that 
\begin{align*}
D\cF_1(\vp)[w](t)&=\int_{\R^2}  (w(s)\th-w(t))    f(\th)ds \Big|^{\th=1}_{\th=-1}\\
&=\int_{\R^2}  (w(s)-w(t))   f(1)ds-\int_{\R^2}  (-w(s)-w(t))   f(-1)ds\\
&=  \int_{\R^2}  (w(s)-w(t))   G_\k(|(t, \vp(t)  )-(s,\vp(s))|)  ds\\
&\quad+\int_{\R^2}  (w(s)+w(t))   G_\k(|(t, \vp(t)  )-(s,-\vp(s))|) ds.
\end{align*}
This gives \eqref{eq:DefcF1-slab-General} after the change of variable $s=t-r$. 
We now have 
\begin{align}
\label{eq:DefcF1-slab}
D\cF_1(\vp)[w](t) 
&=  \int_{\R^2}  (w(t-r)-w(t))   G_\k(|(|r|^2+ ( \vp(t)  - \vp(t-r))^2)^{1/2}|)  ds\nonumber\\
&\quad+\int_{\R^2}  (w(t-r) +w(t)) G_\k(|(|r|^2+ ( \vp(t)  + \vp(t-r))^2)^{1/2}|) ds\nonumber\\
%
%
&=  \int_{\R^2}  \frac{w(t-r)-w(t)}{|r|}   G_{{\k}{|r|}}\left(  \left(1+ \left(\L_0(\vp,t,r)\right)^2\right)^{1/2} \right)  dr\nonumber\\
&\quad+\int_{\R^2}  (w(t-r) +w(t))  G_\k((|r|^2+ ( \vp(t)  + \vp(t-r))^2)^{1/2}) dr,
%
\end{align}
where $\L_0: C^{2,\a}(\R^2)\times \R^4\to \R$ is given by 
$$
\L_0(\vp,t,r)= \int_0^1\n\vp(t-\varrho r )\cdot \frac{r}{|r|} d\varrho.
$$
For every fixed $w\in C^{2,\a}(\R^2)$, we can define $\cJ_i: \cO_\d \to L^\infty(\R^2)$ by 
$$
\cJ_1(\vp )(t):=   \int_{\R^2}  \frac{w(t-r)-w(t)}{|r|}   G_{{\k}{|r|}}\left(\left| \left(1+ \left(\L_0(\vp,t,r)\right)^2\right)^{1/2}\right|\right)  dr
$$
and 
$$
\cJ_2(\vp)(t):= \int_{\R^2}  (w(t-r) +w(t))  G_\k(|(|r|^2+ ( \vp(t)  + \vp(t-r))^2)^{1/2}|) dr.
$$
Therefore by \eqref{eq:DefcF1-slab}
\be\label{eq:cFeq-cJ-is}
 D\cF_1(\vp)[w]= \cJ_1(\vp )+\cJ_2(\vp ).
\ee
Thanks to \eqref{eq:Dk-LT2} and the fact that the integrand in $\cJ_2$ is bounded away from zero, it becomes straightforward to see that  the map $\cJ_2: \cO_\d\to C^{0,\a}(\R^2)$,  is smooth  and for every $k\in \N$, 
 \be \label{eq:Dk-J-i}
\|D^k \cJ_2(\vp)\|\leq c (1+\|\vp\|_{C^{2,\a}(\R^2)})^c\,\|w\|_{C^{2,\a}(\R^2) },
 \ee
 where $c=c(\k,\a,\d,k)>0$. Next, we prove    the regularity of  $\cJ_1: \cO_\d\to C^{0,\a}(\R^2)$.  For $|r|>0$,    we consider the function 
 $$
 h_r(x)= G_{{\k}{|r|}}\left( \left(1+x^2\right)^{1/2} \right).
 $$
 We  observe  for every $k\in \N$, there exists a constant $c=c(\k,k)$ such that 
 $$
 |h^{(k)}_{r}(x)|\leq c\,    G_{{\k}{|r|}}(1/2)\qquad \textrm{ for every $x\in \R$, $r\in \R^{2}\setminus\{0\}$}. 
 $$
 Letting  $\cK_1= C^{2,\a}(\R^2)\times \R^4\to \R$, given by 
 $$
 \cK_1(\vp,t,r):= G_{{\k}{|r|}}\left(\left| \left(1+ \left(\L_0(\vp,t,r)\right)^2\right)^{1/2}\right|\right).
 $$
 We can use \eqref{eq:Faa-de-Bruno}, $g(x)=h_{r}(x)$ and $ T(\vp)= \cK(\vp,t,r)$ to get the following estimates. For every $(t,r)\in \R^2\times \R^2$,
 $$
\| D^k_\vp \cK_1(\vp,t,r)\|\leq c (1+\|\vp\|_{2,\a})^cG_{{\k}{|r|}}(1/2) ,
 $$
 with $c=c(\a,\k,k)>0$. Since also $\int_{\R^2}|r|^{-1} G_{{\k}|r|}(1/2)<\infty$, by  using  \eqref{eq:Dk-LT2}, we get, for every $k\in \N$, 
$$
\|D^k \cJ_1(\vp)\|\leq c (1+\|\vp\|_{C^{2,\a}(\R^2)})^c  \, \|w\|_{C^{2,\a}(\R^2) },
$$
 with $c=c(\k,\a,\d,k)>0$.  This then shows that   $\cJ_1: \cO_\d\to C^{0,\a}(\R^2)$.
 
 From this together with  \eqref{eq:Dk-J-i} and \eqref{eq:cFeq-cJ-is}, we deduce that $D\cF_1: \cO_\d\to \cB(C^{2,\a}(\R^2), C^{0,\a}(\R^2) ) $ is smooth. Hence $\cF$ is smooth in $\cO_\d$ and  
 $$
\|D^k \cF_1(\vp)\|\leq c (1+\|\vp\|_{C^{2,\a}(\R^2)})^c , \qquad  \textrm{ for some positive constant  $c=c(\k,\a,\d,k)$.} 
$$
In addition, we have that 
\begin{align*}
D^k\cF_1(\vp)[w](t) 
&=  \int_{\R^2}  \frac{w(t-r)-w(t)}{|r|}   D^k_\vp \cK_1(\vp,t,r) dr\nonumber\\
&\quad+\int_{\R^2}  (w(t-r) +w(t))  D^k_\vp \cK_2(\vp,t,r) dr,
\end{align*}
where $\cK_2(\vp,t,r)= G_\k(|(|r|^2+ ( \vp(t)  + \vp(t-r))^2)^{1/2}|)$.
 Since $\d$ was arbitrarily chosen, this proves  the smoothness of $\cF_1$ on $\cO$. 

Finally, letting $\vp\equiv \l$ in \eqref{eq:DefcF1-slab}, we obtain \eqref{eq:DcF-lamb-slab}.
\QED
\end{proof}
  We now study the spectral properties of $D\cF(\l)$.
\begin{Lemma}\label{lem:def-linearized-eigen-slab}
 For every $\l>0$ we define the linear operator 
 \be \label{eq:def-L-lamb-slab}
 L_\l:=D\cF(\l) : C^{2,\a}_{p,\cR}\to C^{0,\a}_{p,\cR}.
 \ee
 Then the functions $e_k\in C^{2,\a}_{p,\cR}$, given by 
 \be \label{eq:eigen-fonc-slab}
 e_k(t):=\frac{1}{\pi}\prod_{i=1}^2\cos(k_i t_i),
 \ee
 for $k\in \N^2$,
 are the    eigenfunctions of $L_\l$ and moreover
 \be \label{eq:Ll-eigen-slab}
 L_\l(e_k)=\s_{\l,\g}(|k|) e_k,
 \ee 
where $|k|=\sqrt{k_1^2+k_2^2}$ and  for every $\ell\in \N$, 
\begin{align}\label{eq:eqgien-slab}
\s_{\l,\g}(\ell)
  &= \ell^2 - 2\pi \g  \left\{  \frac{1}{\k}-  \frac{1}{\sqrt{\k^2+\ell^2  }} -  \frac{\exp(-2\l\sqrt{\k^2+\ell^2  })}{  \sqrt{\k^2 +\ell^2 }} -  \frac{\exp(-2\l\k )}{  \k} \right\}  .
\end{align}
 In particular,

  \be\label{eq:monoton-lambda-slab} 
   \de_\l   \s_{\l,\g}(1)<0   \qquad\textrm{  for every $\g,\l>0$  }
   \ee
   and  
 \be\label{eq:asymp-sigk-slab}
 \lim_{\ell\to \infty}  \frac{ \s_{\l,\g}(\ell)}{\ell^2}= 1.
 \ee
\end{Lemma}
\begin{proof}
  By direct computations using Fourier decomposition, we have 
 \begin{align}\label{eq:exp-intial-sigma-l-g}
 \s_{\l,\g}(|k|)
  &= |k|^2 - \g    \int_{\R^2} \left(1- \prod_{i=1}^2\cos(k_ir_i) \right) G_\k\left(|r| \right)dr \nonumber\\
  &\quad + \g      \int_{\R^2}\left(1+\prod_{i=1}^2 \cos(k_i r_i)\right) G_\k\left((|r|^2+4 \l^2 )^{1/2} \right) dr   .
\end{align} 
By \eqref{eq:int-cos-G-k-sqrt}, we have 
$$
\int_{\R}\cos(k_2 r_2) \int_{\R} \cos(k_1 r_1) G_\k\left((r^2_1+r_2^2)^{1/2} \right)dr_1 dr_2= 2\int_{\R}\cos(k_2 r_2) K_0(r_2\sqrt{\k^2+k^2_1})dr_2
$$
and by \eqref{eq:int-cos-K0}
$$
\int_{\R}\cos(k_2 r_2) K_0(r_2\sqrt{\k^2+k^2_1})dr_2= \frac{2\pi}{2\sqrt{\k^2+|k|^2  }}.
$$
We then deduce that 
$$
\int_{\R}\cos(k_2 r_2) \int_{\R} \cos(k r_1) G_\k\left((r^2_1+r_2^2)^{1/2} \right)dr_1 dr_2= \frac{2\pi}{\sqrt{\k^2+|k|^2  }}.
$$
Moreover, taking $k=0$ in the above expression, we have 
$$
 \int_{\R^2}   G_\k\left(|r|\right)dr=\frac{2\pi}{\k}.
$$
From the  two qualities above, we deduce that
\be\label{eq:frst-exp} 
 \int_{\R^2}\left(1- \prod_{i=1}^2 \cos(k_ir_i)\right) G_\k\left(|r| \right)dr=\frac{2\pi}{\k}-  \frac{2\pi}{\sqrt{\k^2+|k|^2  }}.
\ee
By \eqref{eq:int-cos-G-k-sqrt} and \eqref{eq:int-cos-K0-sqrt}, we get
 \be\label{eq:sec-exp} 
 \int_{\R^2}\prod_{i=1}^2 \cos(k_i r_i) G_\k\left((|r|^2+4 \l^2 )^{1/2} \right) dr=  2 \pi\frac{\exp(-2\l\sqrt{|k|^2+\k^2})}{\sqrt{|k|^2+\k^2}}
 \ee
and letting $k=0$, we have 
$$
  \int_{\R^2}  G_\k\left((|r|^2+4 \l^2 )^{1/2} \right) dr =  2 \pi\frac{\exp(-2\l \k) }{ \k}. 
$$
Using this with \eqref{eq:frst-exp} and \eqref {eq:sec-exp}  in  \eqref{eq:exp-intial-sigma-l-g}, we get the desired expression of $\s_{\l,\g}(|k|)$ in \eqref{eq:eqgien-slab}.
\QED
\end{proof}
Next, we give some qualitative properties of $\s_{\l,\g} $ .

\begin{Lemma}\label{lem:Lambda-star-slab}
Let  $\k>0$ and let $\g$ be such that     
\be \label{eq:restric-gamma-slab}
 \frac{1}{2\pi} \frac{1}{  \frac{1}{\k}-  \frac{1}{\sqrt{\k^2+1 }}  }<  \g <  \frac{1}{2\pi} \frac{   1}{  \frac{1}{(\k^2+1 )^{3/2}}  }.
\ee
Then 
there exists a  unique $\l_*=\l_*(\g,\k)>0$ such that
  \be \label{eq:s-eq0-slab}
  \s_{\l_*,\g}(1)=0 ,
  \ee
\be \label{eq:sig-l-st-g-at-o-posi}
\s_{\l_*,\g}(0)=\frac{4\pi\g \exp(-2\l_* \k)}{\k}>0
\ee
and 
\be \label{eq:sk-neq-0-slab}
 \s_{\l_*,\g}(\ell+1)> \s_{\l_*,\g}(\ell) \qquad \textrm{ for every $\ell\geq 1  $}.
 \ee
\end{Lemma}
\begin{proof}
We first note that, by elementary calculations, the strict inequality
$$
\frac{1}{  \frac{1}{\k}-  \frac{1}{\sqrt{\k^2+1 }}  }<  \frac{   1}{  \frac{1}{(\k^2+1 )^{3/2}}  } 
$$
 is equivalent to 
$$
(\k^2+1)^2-(\k^4+2\k^2)>0,
$$
which holds true for every $\k>0$. Therefore we can always pick $\g$ as in \eqref{eq:restric-gamma-slab}.\\

In view of \eqref{eq:eqgien-slab}, we have 
 $$
\s_{\l,\g}(1)
  = 1 - 2\pi \g  \left\{  \frac{1}{\k}-  \frac{1}{\sqrt{\k^2+1  }} - \frac{\exp(-2\l\sqrt{\k^2+1}) }{ \sqrt{\k^2+1}} - \frac{\exp(-2\l\k) }{ \k}\right\}  .
  $$
It is easy to see that 
$$
\lim_{\l\to 0}\s_{\l,\g}(1)= 1+  \frac{4\pi \g }{\sqrt{\k^2+1  }}>0.
$$
Now if $\g$ satisfies the lower bound in  \eqref{eq:restric-gamma-slab}, we have that  
$$
 \lim_{\l\to +\infty}\s_{\l,\g}(1)= 1 - 2\pi \g  \left\{ \frac{1}{\k}-   \frac{1}{\sqrt{\k^2+1  }}   \right\}  <0.
$$ 
Therefore  \eqref{eq:s-eq0-slab} follows, for  a unique $\l_*=\l_*(\g,\k)$ thanks to  \eqref{eq:monoton-lambda-slab}. We now prove  \eqref{eq:sk-neq-0-slab}, where the assumption  on the  upper bound in \eqref{eq:s-eq0-slab} will be used. 

Indeed,  we consider the smooth function $g_{\l,\g}:[1,\infty)\to \R$, given by
$$
g_{\l,\g}(x)=  x^2 - 2\pi \g  \left\{  \frac{1}{\k}-  \frac{1}{\sqrt{\k^2+x^2  }} - \frac{\exp(-2\l\sqrt{\k^2+x^2}) }{ \sqrt{\k^2+x^2}} - \frac{\exp(-2\l\k) }{ \k}\right\}  .
$$
Then 
$$
g'_{\l,\g}(x)=2x\left (1-  \frac{\pi \g }{(\k^2+x^2  )^{3/2}}  -   \frac{\pi \g \exp(-2\l\sqrt{\k^2+x^2})}{(\k^2+x^2   )^{3/2}}-  \frac{2\pi \g \l \exp(-2\l\sqrt{\k^2+x^2})}{\k^2+x^2    } \right).
$$
We want $g'_{\l,\g}(x)>0$   for every $\l>0$ and $x\geq 1$, which holds true  provided
$$
 \pi \g <\frac{   1}{\sup_{\stackrel{\l>0}{ x\geq 1}}  \left(\frac{1 }{(\k^2+x^2  )^{3/2}}  +   \frac{  \exp(-2\l\sqrt{\k^2+x^2})}{(\k^2+x^2   )^{3/2}}+  \frac{  2\l \exp(-2\l\sqrt{\k^2+x^2})}{\k^2+x^2    }  \right)}.
$$
But   since the (maximized) function in the denominator  is decreasing in $x$ and $\l$, this is equivalent to 
$$
 \pi \g <\frac{   1}{  \frac{2}{(\k^2+1 )^{3/2}}  }
$$
which is our assumed  upper bound   in  \eqref{eq:restric-gamma-slab}.
It is now clear hat  as long as  $\g$ satisfies  \eqref{eq:restric-gamma-slab},   we have  that  $g'_{\l,\g}(x)>0$  for every $\l>0$ and $x\geq 1$.   This  completes the proof of   \eqref{eq:sk-neq-0-slab}  since  $\s_{\l,\g}(\ell)=g_{\l,\g}(\ell)$ for every $\ell\geq 1$.
\QED
\end{proof}
We will consider the subspaces of  Sobolev spaces which are invariant under the rotations   
$$
\cR(2):=\left\{\left( \begin{array}{cc}
0&1\\
1&0
\end{array}\right), 
\left( \begin{array}{cc}
-1&0\\
0&1
\end{array}\right)   \right\}. 
$$
In the following, for $j\in \N$, 
we define
\begin{equation*}
  \label{eq:def-hpe}
H^j_{p,\cR} := \Bigl \{v \in H^{j}_{loc}(\R^2) \::\: \text{$v$ is 
$2\pi\Z^2$-periodic  and $v(M t)=v(t)$,  $\,M\in\cR(2),\, t\in \R^2$}\Bigl \}  
\end{equation*}
and we put $L^2_{p,\cR}:= H^{0}_{p,\cR}$. Note that $L^2_{p,\cR}$ is a Hilbert space with scalar product
$$
(u,v) \mapsto \langle u,v \rangle_{L^2([-\pi,\pi]^2)} := \int_{[-\pi,\pi]^2} u(t)v(t)\,dt \qquad \text{for $u,v \in L^2_{p,\cR}$.}
$$
 Then   the functions $e_{k}$ in \eqref{eq:eigen-fonc-slab} form an orthonormal basis for $L^2_{p,\cR}$. Moreover, 
$H^j_{p,\cR} \subset L^2_{p,\cR}$ is characterized as the subspace of all functions $v \in L^2_{p,\cR}$ such that
$$
\sum_{k \in \N  ^2} (1+|k|^2)^{j} \langle v, e_k \rangle_{L^2([-\pi,\pi]^2)} ^2 < \infty.
$$
Therefore, $H^j_{p,\cR}$ is also a Hilbert space with scalar product
\begin{equation*}
  \label{eq:scp-hj}
(u,v) \mapsto \langle u,v \rangle_{H^j_{p,\cR}} :=  \sum_{k \in \N^2} (1+|k|^2)^{j} \langle u, e_k \rangle_{L^2([-\pi,\pi]^2)}  \langle v, e_k \rangle_{L^2([-\pi,\pi]^2)}  \qquad \text{for $u,v \in H^j_{p,\cR}$.}
\end{equation*}
We  consider   the eigenscapces corresponding to $\s_{\l,\g}(\ell)$, given by  
\begin{equation}
  \label{eq:def-Vell}
V_\ell:= \textrm{span}\{ e_k \;:\; |k|=\ell \}  
\end{equation}
 and $P_\ell: L^2_{p,\cR}\to  V_\ell $ be the $L^2$-orthogonal projection on $V_\ell$. We will  denote by $V_\ell^\perp$ to $L^2$-complement of $V_\ell$, which is given by
 $$
 V_\ell^\perp:=\{v\in V_\ell\,:\, P_\ell v=0\}.
 $$
Obviously, the direct-sum  $ \bigoplus_{\ell\in\N}V_\ell$ provides  a complete   decomposition  of $L^2_{p,\cR}$. More precisely for any $v\in L^2_{p,\cR}$, we can write it as 
$$
v=\sum_{\ell\in \N} P_\ell v.
$$
 Moreover, by \eqref{eq:Ll-eigen-slab}, we have  
\be\label{eq:-Lv-Fpurier-slab}
 L_\l(v) =\s_{\l,\g}(\ell) v \qquad \textrm{ for all $v\in V_\ell$}. 
\ee
It is easy to see that 
\be\label{eq:V_1-cap-XcP}
 V_1\cap C^{0,\a}_{p,\cR} =\textrm{span}\{\ov{v}\}=\R \ov v\qquad   \textrm{ with }\,\,\,\ov{v}(t_1,t_2)= \cos(t_1) + \cos(t_2).
\ee
We also consider the nonlinear operator $\cF$ defined in \eqref{eq:dexp-cF-slab}, and we note that $\cF$ maps $\cO \cap C^{2,\a}_{p,\cR}$ into $C^{0,\a}_{p,\cR}$. Consider the open set
$$
\cU:= \{(\lambda,u) \in \R \times C^{2,\a}_{p,\cR} \::\: \lambda>0, \: u> - \lambda \}  \subset \R \times C^{2,\a}_{p,\cR}.
$$
The proof of  our main theorem will be completed by applying the Crandall-Rabinowitz bifurcation theorem to the smooth nonlinear operator
\be \label{eq:def--G}
\cG: \cU \subset \R \times C^{2,\a}_{p,\cR}  \to C^{0,\a}_{p,\cR}, \qquad \cG(\lambda,u) =
\cF(\lambda+u)-\cF(\l).
\ee

We have the following

\begin{Proposition}\label{propPhi}
Let $\g$ satisfy \eqref{eq:restric-gamma-slab}. 
Then there exists a unique $\lambda_*=\l_*(\g,\k)>0$ such that  
the linear operator $  L_{\l_*}  = D_u \cG(\l_*,0) $
has the following properties.
\begin{itemize}
\item[(i)] The kernel $N(L_{\l_*} )$ of $L_{\l_*} $ is  given by 
  \begin{equation}
    \label{eq:def-v-0}
N(L_{\l_*} )=V_1=\R\ov{v} , \qquad \ov{v}(t_1,t_2)= \cos(t_1) + \cos(t_2).
  \end{equation}
\item[(ii)] The range $R(L_{\l_*} )$ of $L_{\l_*} $ is given by
$$
R(L_{\l_*} )=C^{0,\a}_{p,\cR}\cap V_1^\perp .
$$
%
\item[(iii)]Finally, 
\begin{equation}
  \label{eq:transversality-cond}
\partial_\lambda \Bigl|_{\lambda= \lambda_*} L_\lambda\ov{v}=\partial_\lambda \Bigl|_{\lambda= \lambda_*} \s_{\l,\g}(1)\ov{v}  \;\not  \in \; R(L_{\l_*}).
\end{equation}
\end{itemize}

\end{Proposition}

\proof
For $\g$ satisfying  \eqref{eq:restric-gamma-slab}, by Lemma~\ref{lem:Lambda-star-slab}, there exists a unique $\lambda_*=\l_*(\g,\k)>0$ such that $\sigma_{\lambda_*,\g}(1) =0$ and $\sigma_{\lambda_*,\g}(\ell) >0 $ for all $\ell\in \N\setminus\{1\}$. Hence by    \eqref{eq:-Lv-Fpurier-slab}, we have that   $L_{\l_*} v=0$ if and only if $v\in V_1$. This with \eqref{eq:V_1-cap-XcP} prove (i).\\

We now prove   (ii). We  pick $f \in C^{0,\a}_{p,\cR}\cap V_1^\perp\subset L^2_{p,\cR}\cap V_1^\perp$. Then using Fourier decomposition,  \eqref{eq:-Lv-Fpurier-slab},    \eqref{eq:asymp-sigk-slab} and \eqref{eq:sk-neq-0-slab}, we have a unique $w\in H^2_{p,\cR}\cap V_1^\perp$,   such that 
$$
L_{\l_*} w=f \qquad \textrm{ in $\R^2$}.
$$ 
 The proof of (ii) finishes once we show that $w\in  C^{2,\a}_{loc}(\R^2)$.

 By Morrey's embedding, $w\in  H^2_{loc}(\R^2)\subset C^{0,\a}_{loc}(\R^2)$, for every $\a\in (0, 1)$.   We then have
\be \label{eq:L-star-w-eq-f-slab}
-\D w=\g  w(t)\int_{\R^2}   G_\k\left(|r| \right)dr- \g   w(t) \int_{\R^2}  G_\k\left((|r|^2+4 \l^2 )^{1/2} \right) dr- F(t),
\ee
where 
 %
  %
%
%
$$
F(t):=\int_{\R^2} w(t-r) G_\k\left(|r| \right)dr+\int_{\R^2} w(t-r) G_\k\left((|r|^2+4 \l^2_* )^{1/2} \right) dr  .
$$
Then  for every $t\in \R^2$
\begin{align*}
|F(t)|&\leq  \|w\|_{C^{0,\a}(\R^2)}\left( \int_{\R^2}  G_\k\left(|r| \right)dr+\int_{\R^2}   G_\k\left((|r|^2+4 \l^2_* )^{1/2} \right) dr  \right)  \\
&\leq  C \|w\|_{C^{0,\a}(\R^2)}.
\end{align*}
It is also easy to see that for $t, \ov t\in \R^2$
\begin{align*}
|F(t)-F(\ov t)|& \leq C\|w\|_{C^{0,\a}(\R^2)}  |t-\ov t|^\a .
\end{align*}
We then obtain that $F\in C^{0,\a}(\R^2)$, for every $\a\in (0,1)$.  Now  by \eqref{eq:L-star-w-eq-f-slab} and  standard elliptic regularity theory, $w\in C^{2,\a}_{loc}(\R^2)$ for every  $\a\in (0,1)$.   Hence (ii) follows.\\
It remains to prove~\eqref{eq:transversality-cond}, which follows
from \eqref{eq:monoton-lambda-slab}  and the identity
\begin{equation}\label{Dlambda}
\partial_\lambda \Bigl|_{\lambda= \lambda_*} L_\lambda \ov{v} =  \partial_\lambda \Bigl|_{\lambda= \lambda_*}\sigma_{\l,\g}(1)\ov{v} .
\end{equation}
\QED
\subsection{Proof of Theorem \ref{th:Slab}}
Thanks to \eqref{propPhi}, we can apply 
  the Crandall-Rabinowitz Theorem (see \cite[Theorem 1.7]{M.CR}), to have
\begin{Proposition}\label{prop:exist-slab-CR} 
There exists  ${s_0}>0$, only depending on $ \k$ and $\g$, and a smooth curve
\be \label{eq:exist-branch-slab}
(-{s_0},{s_0}) \to (0,\infty)\times C^{2,\a}_{p,\cR}, \qquad s \mapsto (\lambda_s,u_s)
\ee
such that
\begin{enumerate}
\item[(i)] For every  $s \in (-{s_0},{s_0})$ 
 \be\label{eq:solved-slabl-l-s--u-s} \cG(\lambda_s, u_s)= \cF(\l_s+ u_s)-\cF(\l_s)=0   . \ee
\item[(ii)] $\lambda_0= \lambda_*$.
\item[(iii)]  $u_s = s \bigl( \ov{v}  + v_s\bigr)$ for all $s \in (-s_0,s_0)$,  with a smooth curve
$$
(-{s_0},{s_0}) \to C^{2,\a}_{p,\cR}, \qquad s \mapsto v_s
$$
satisfying $v_0=0$  and $\int_{[-\pi, \pi]^2}v_s(t)\ov{v}(t)dt=0$.
\end{enumerate}
\end{Proposition}
%

%
%
\section{Periodic lamellae and proof of Theorem \ref{th:Lamellae}}
For $\e\in \R\setminus\{0\}$, with $|\e|$ small, we define 
$$
\cL_{\psi}^\e=\O_{\psi}+\frac{1}{|\e|} \Z e_3=\bigcup_{p\in \Z}\left( \O_{\psi}+\left(0,0,\frac{p}{|\e|}\right)\right),
$$
where as above $\psi\in C^{2,\a}_{p,\cR}\cap \cO$.
 For $x\in \{-1,1\} $, the mean curvature takes the form
\be
H_{\de \cL_{\psi}^\e}((t, \psi(t)x) )= -\div  \frac{\n \psi(t)}{\sqrt{1+|\n \psi(t)|^2}} ,
\ee
and by a change of variable, the nonlocal part takes the form
\begin{align*}
\int_{\cL_{\psi}^\e }&G_\k(|(t, \psi(t)x)-\xi|) d\xi
=  \int_{\R^2} \int_{-1}^1\psi(s) G_\k(|(t, \psi(t)x)-(s,\psi(s)y) |) dy ds \\
&\quad + \sum_{p\in \Z _*} \int_{\R^2} \int_{-1}^1\psi(s) G_\k(|(t, \psi(t)x)-(s,\psi(s)y)-(0,0,{|\e|}^{-1}
{p})|) dy ds.
\end{align*}
Now using the fact that $ \Z _*=-\Z _*$ and the change of variable $y$ to $-y$, we see that 
$$
\int_{\cL_{\psi}^\e }G_\k(|(t,- \psi(t))-\xi|) d\xi=\int_{\cL_{\psi}^\e }G_\k(|(t, \psi(t))-\xi|) d\xi.
$$
This implies, in particular,  that for every $x\in \{-1,1\}$,  $t\in \R^2$ and $|\e|<\frac{1}{2\l}$,
$$
H_{\de \cL_{\psi}^\e}((t, \l x) )= \g   \l \sum_{p\in \Z } \int_{\R^2} \int_{-1}^1  G_\k\left(( |r|^2+ (\l(1-y)-{|\e|}^{-1}
{p} )^2  )^{1/2} \right) dy dr.
$$
Namely $\cL_{\l}^\e$ is an equilibrium pattern  for every $\l>0$,  with $|\e|<\frac{1}{2\l}$.

Recalling  \eqref{eq:dexp-cF-slab}, we get 
\begin{align}\label{eq:def-cH-lamellae}
& \cH_{ \cL_{\psi}^\e}((t, \psi(t)x) )=H_{\de \cL_{\psi}^\e}((t, \psi(t)) )+ \int_{\cL_{\psi}^\e }G_\k(|(t, \psi(t))-\xi |) d\xi \nonumber  \\
&      =\cF(\psi)(t)  + \sum_{p\in \Z _*} \int_{\R^2} \int_{-1}^1\psi(s) G_\k\left(\left|(t, \psi(t))-(s,\psi(s)y)-(0,0,{|\e|}^{-1}
{p})\right|\right) dy ds,
\end{align}
where  $\cF: \cO\to C^{0,\a}(\R^2)$ is as in the previous section, is given by
\begin{align}\label{eq:def-cF-lamellae}
&\cF(\psi)(t):= -\div \frac{\n \psi(t)}{\sqrt{1+|\n \psi(t)|^2}}   +\g  \int_{\R^2}\int_{-1 }^1\psi(s)G_\k(|(t, \psi(t) )-(s,\psi(s)y)|) dyds .
\end{align}
We now 
define   $\widehat \cF:\R\setminus\{0\}\times \cO\to C^{0,\a}(\R^2)$   by 
\begin{align}\label{eq:def-hat-F-e-psi}
\widehat \cF(\e,\psi )(t):=\sum_{p\in \Z _*} \int_{\R^2} \int_{-1}^1\psi(s) G_\k\left(\left|(t, \psi(t))-(s,\psi(s)y)-(0,0,{|\e|}^{-1}
 {p} )\right|\right) dy ds .
\end{align}
By similar arguments as in the proof of Lemma \ref{lem:Diff-of-cF-slab}, we can deduce    that 
\begin{align*}
\label{eq:DefcF1-slab-General}
&D \widehat \cF(\e,\psi ) [w](t) \\
&= \sum_{p\in \Z_*} \int_{\R^2}  (w(t-r)-w(t))   G_\k((|r|^2+ ( \psi(t)  - \psi(t-r)+|\e|^{-1}p)^2)^{1/2})  dr\nonumber\\
& + \sum_{p\in \Z_*} \int_{\R^2}  (w(t-r)  +w(t))G_\k((|r|^2+ ( \psi(t)  + \psi(t-r)+|\e|^{-1}p)^2)^{1/2}) dr .
\end{align*}
Let $\cU_\l$ be a neighborhood  of $\l>0$ in $\cO$. Then there exists $\ov{\e},c>0$, only depending on $\cU_\l$ and $\k$, such that for 
  $|\e|\in (0, \ov{\e})$   and $\psi\in \cU_\l$, we have  
\begin{align*}
G_\k((|r|^2+ ( \vp(t)  - \vp(t-r)+|\e|^{-1}p)^2)^{1/2}) &\leq   G_\k\left( \left( |r|^2+c^2 {|\e|}^{-2}
 {p^2} \right)^{1/2} \right).
 %
\end{align*}
Therefore by a change of variable and \eqref{eq:int-G-k-R2-sqrt}, if  $|\e|\in (0, \ov{\e})$   and $\psi\in \cU_\l$, then 
\begin{align*}
\| D \widehat\cF(\e, \vp)[w]\|_{C^{0,\a}(\R^2)} &\leq 4\|w\|_{C^{0,\a}(\R^2) }\sum_{p\in \Z_*}    \int_{\R^2}G_\k\left( \left( |r|^2+c^2 {|\e|}^{-2}
 {p^2} \right)^{1/2} \right) dr\\
&\leq  4 c\|w\|_{C^{0,\a}(\R^2) } \sum_{p\in \Z_*}  |\e|^{-1}|p|    \int_{\R^2}G_{\k c\frac{|p|}{ |\e|}}\left( \left( |t|^2+1 \right)^{1/2} \right) dt \\
&\leq  4 c\|w\|_{C^{0,\a}(\R^2) }  |\e| \sum_{p\in \Z_*}  |p|^{-1}    \exp\left(-\k c {|p|}{ |\e|^{-1}} \right) \\
&\leq  4 c\|w\|_{C^{0,\a}(\R^2) }|\e|  \exp\left(-\frac{\k}{2} c { |\e|^{-1}} \right) .
\end{align*} 
Up to decreasing $\ov{\e}$ and choosing a smaller neighborhood of $\l$, we can easily see that 
  $\widehat \cF$ extends to a  smooth function $(-\ov{\e},\ov{\e})\times \cU_\l \to C^{0,\a}(\R^2)$ with $D^k\widehat \cF(0,\psi )=0$ for all $\psi\in \cU_\l$ and $k\in \N$.
In view of \eqref{eq:def-cF-lamellae},\eqref{eq:def-hat-F-e-psi} and \eqref{eq:def-cH-lamellae}, we define  $$\cN_\l: (-\ov{\e},\ov{\e})  \times  \cU_\l\cap C^{2,\a}_{p,\cR}\to  C^{0,\a}_{p,\cR}$$  by
\be \label{eq:def-cN-lamel}
   (\e,\psi)\mapsto \cN_\l(\e,\psi)(t):= \cF(\psi)(t)+\g\widehat\cF(\e,\psi)(t)=\cH_{ \cL_{\l }^\e}((t, \psi (t)) ),
\ee
In the following, we let   $\g$ satisfy  \eqref{eq:restric-gamma-slab},

Instead of applying the Crandall-Rabinowitz Theorem  \cite[Theorem 1.7]{M.CR}, we will use the same argument  of its proof, mainly based on the Implicit Function Theorem. More precisely, 
we  apply the latter theorem to solve the equation $\cM=0$, with  the function 
$$
\cM: (-\ov{\e},\ov{\e})\times (- \l_* /8, \l_* /8) \times (1/2,3/2)  \times    C^{2,\a}_{p,\cR}\to  C^{0,\a}_{p,\cR}, 
$$
  given by
\begin{align}\label{eq:def-cM-lamell}
 \cM(\e,s,\d, v)(t)&:= \frac{\cN_{\l_*}(\e, \d \l_*+ s (\ov{v}+v))(t)-\cN_{\l_*}(\e, \d \l_* ) }{s}\\
 &=\int_0^1D_\psi\cN_{  \l_*}(\e, \d\l_*+\varrho s(\ov{v}+v)) )[\ov{v}+v ](t)d\varrho. \nonumber
\end{align}
It is clear that $\cM$ is smooth in $(-\ov{\e},\ov{\e})\times    (- \l_* /8, \l_* /8) \times  (1/2,3/2) \times B_{C^{2,\a}_{p,\cR}}(0, 1 /8)$. Now 
from Proposition \ref{propPhi}(i), we have that $\cM(0,0,1,0)= L_{\l_*} (\ov{v})=0$.  
We also have that 
 $$ D_v\cM(0,0,1,0) =  L_{\l_*} : C^{2,\a}_{p,\cR}\cap V_1^\perp\to C^{0,\a}_{p,\cR}\cap V_1^\perp$$  is an isomorphism  by Proposition \ref{propPhi}(ii). Moreover,  by  Proposition \ref{propPhi}(iii),  $$   D_\d  \cM(0,0, 1,0): \R  \to   V_1$$ is an isomorphism. From these, we can now apply the implicit function theorem to get $s_0,\e_0 >0$ and  a smooth curve
 $$
 (-\e_0,\e_0)\times   (-s_0,s_0)\to   (1/2,3/2)  \times   C^{2,\a}_{p,\cR}\cap V_1^\perp,\qquad ( \e,s)\mapsto (\d_{\e,s}, v_{\e,s})
 $$
with $ \d_{0,0}=1$ and $v_{0,0}=0$ such that 
 $$
  \cM(\e, s, \d_{\e,s},v_{\e,s})=0 \qquad \textrm{ on $C^{0,\a}_{p,\cR}$}.
  $$
From this together with \eqref{eq:def-cM-lamell} and \eqref{eq:def-cN-lamel}, we conclude that, for every $t\in \R^2$, 
\begin{align} \label{eq:solved-eq-Lamellae}
 \cF(\l_{\e,s}+ s (\ov{v}+v_{\e,s}))(t)+\g\widehat\cF(\e,\l_{\e,s}+ s (\ov{v}+v_{\e,s}))(t)&=  \cF(\l_{\e,s})+\g\widehat\cF(\e, \l_{\e,s})
\end{align}
where we have set $\l_{\e,s}:= \d_{\e,s}\l_*$. This ends the first part of the proof of  Theorem \ref{th:Lamellae}.  It follows from the uniqueness, inherent to the construction, that $\l_{0,s}=\l_s $ and $v_{0,s}= v_s $, for all $s\in (-s_0,s_0)$, where $\l_s$ and $v_s$ are given by Theorem \ref{th:Slab}.

 \QED

 \subsubsection{Proof of Theorem \ref{th:Lamellae} (completed)}
 The existence of the smooth curve of equilibrium patterns   $\cL_{\psi_{\e,s}}^\e$  follows from  the identity \eqref{eq:solved-eq-Lamellae}. Obviously $\psi_{0,s}=\vp_s$ as a consequence of the uniqueness form the Crandall-Rabinowitz theorem. 
 \QED

\section{Periodic lattice of near-cylinders  and near-spheres  }\label{s:latt-Cyl-Sph}
\subsection{Some preliminary and notations}\label{ss:prel-not}
Let $M\in\N$ with $1\leq M\leq N$.  Let  $\left\{ \textbf{a}_1;\dots;\textbf{a}_M\right\}$  be a basis of $\R^M$.  
 We recall   from the first section the    $M$-dimensional  {Bravais lattice} in $\R^N$:
\be 
\scrL^M=\left\{\sum_{i=1}^N k_i \textbf{a}_i\,:\,  k=(k_1,\dots,k_M)\in \Z^N\right\}\subset \R^N.
\ee 
%



%
%
Throughout this paper, we will put  $\scrL_*^M:= \scrL^M \setminus \{0\}$ and   $\Z^M_*:= \Z^M \setminus \{0\}$.

We let  $S^{N-1}$ be the unit sphere on $\R^N$.  We
denote by $\cS_k$  the finite dimensional subspace of spherical harmonics of degree $k$ on $S^{N-1}$, and by $(Y^i_{k})_{k\in\N,\, i=1,\dots, n_k}$ an orthonormal basis of $\cS_k$.  It is well known that  $n_0=1$ with $Y^1_0(\th)=\frac{1}{\sqrt{|S^{N-1}|}}$ and $n_1=N$ with $Y^i_1(\th)=\frac{\th_i}{\sqrt{|B^N|}}$, for each $i=1,\dots,N$.  We recall that $-\D_{S^{N-1}} Y = k(k+N-2)  Y $, for every $Y\in \cS_k$.  Moreover for a  function $f: (0,2)\to \R_+$, with $\int_{0}^2f(t)(t(2-t))^{\frac{N-3}{2}}dt<\infty$,  by the Funck-Hecke formula, see\cite[page 247]{El},  for $k\geq1$ and $\th\in S^{N-1}$,  we have 
$$
\int_{S^{N-1}}(Y(\th)-Y(\s))f(|\th-\s|)\, d\s=\mu_k Y(\th), \qquad\textrm{ for every $Y\in \cS_k$},
$$
for some real number $\mu_k\geq0$.
From this, we deduce that 
$$
\mu_k=\frac{1}{2}\int_{S^{N-1}}\int_{S^{N-1}}(Y^i_{k}(\th)-Y^i_{k}(\s))^2f(|\th-\s|)\, d\s d\th  \qquad \textrm{ for every $i=1,\dots, n_k$}. 
$$
For a function $w\in L^2(S^{N-1} )$, we can decompose it as
$$
w=\sum_{k\in \N}\sum_{i=1}^{n_k} w_k^i  Y_k^i,\qquad \textrm{ with } \qquad   w_k^i=\int_{S^{N-1}}w(\s) Y_k^i(\s)\, d\s.
$$
For $j\in \N$,  we define the Sobolev  spaces 
\be
H^j(S^{N-1} )=\left\{w\in L^2(S^{N-1})\,:\,  \sum_{k\in \N}\sum_{i=1}^{n_k}(1+|k|^2)^{{j}} (w_k^i )^2 <\infty  \right\}
\ee
and
\be
H^j_{\textbf{e}}(S^{N-1} )=\{w\in  H^j(S^{N-1} )\,:\, \textrm{ $w$ is even}    \}. 
\ee
  We define also    H\"older spaces
$$
C^{j,\alpha}_{\textbf{e}}(  S^{N-1}):= \biggl\{  \phi \in C^{j,\alpha}(  S^{N-1}), \quad \textrm{$\phi(-\th)=\phi(\th)$      }\biggl\}
$$
 endowed with  usual H\"older norm of $C^{j,\alpha}(  S^{N-1})$.

Following \cite{Cabre2015B},  for a fixed  $e \in S^{N-1} $ and a Lipschitz continuous map of rotations $S^{N-1}  \mapsto SO(N)$, $\th \mapsto R_\theta$ with the property that 
\begin{equation}
  \label{eq:def-Se}
S_e:=  \{\th \in S^{N-1}\::\: \th \cdot e \ge 0\}  \subset  \{\th \in S^{N-1} \::\: R_{\th}e = \th \}.
\end{equation}

\begin{Example}[\cite{Cabre2015B}]
\label{example-rotations}  
Consider the map $\th \mapsto R_\theta$ defined as follows. For $\theta \in S^{N-1} $ with $\th \cdot  e \ge 0$, we let $R_\theta$ be the rotation of the angle $\arccos \th \cdot e$ which maps $e$ to $\th$ and keeps all vectors perpendicular to $\th$ and $e$ fixed. We then extend the map $\th \mapsto R_\theta$ to all of $S^{N-1} $ as an even map with respect to reflection at the hyperplane $\{\th \in \R^N \::\: \th \cdot e = 0\}$.
\end{Example}
It is easy to see that 
\be\label{eq:Rota-isom}
|R_{\th}\s-\th |=|e-\s| \qquad \textrm{ for all $\th\in S_e$ and $\s\in S^{N-1} $}. 
\ee
Moreover, the Lipschitz property of the map $\th \mapsto R_\th$ implies that there is a constant $C>0$ with 
\be \label{eq:R-Lipschitz}
\|R_{\th_1} - R_{\th_2}\| \le C |\th_1-\th_2| \qquad \text{for all $\th_1,\th_2 \in S^{N-1} $,} 
\ee 
where, here and in the following, $\|\cdot\|$ denotes the usual operator norm.
We put
$$\cD:= \{\phi \in C^{2,\a}(S^{N-1}):\: \phi>0\}  .$$
For every $\vp\in \cD$,  we let  
\be\label{eq:def-F_vp-paramet} 
F_\vp : \R_+ \times S^{N-1} \to \R^N,\qquad F_\vp(r ,\s):= r \s \vp (\s ).
\ee
The map $F_\vp$ defines  a parameterization of  the set 
\be 
B_\vp^N:= F_{\vp}([0,1) \times S^{N-1}).
\ee 
We will denote $S^{N-1}_\vp:= \de B_\vp^N= F(1,S^{N-1} )$.
See e.g. \cite{Cabre2015B}, the unit outer normal   of $B_\vp^N$ at a point $F_\vp(1,\sigma)$, $\sigma \in S^{N-1}$, is given by
\be \label{eq:normal-S-vp}
\nu_\vp(F_\vp(1,\s)) = \frac{\phi(\s)\s  -\n \vp(\s)}{\sqrt{\phi^2(\s) + |\n\vp(\s)|^2}}
\ee
and for every continuous function $f$ on $\R^N$, we have
\begin{equation}
  \label{eq:transformation-rule}
\int_{S^{N-1}_\vp} f(y) \,dV(y) =  \int_{S^{N-1}} [f\circ  F_\vp(1,\cdot )](\s)  \G_\vp(\s)  \, d\s \qquad \text{with}\quad  \G_\vp=    \vp^{N-2}\sqrt{ \vp^{2}+  |\n \vp|^2}.
\end{equation}
Moreover letting
$$
\cA(\vp)= \int_{S^{N-1}}   \G_\vp(\s)  \, d\s,
$$
then differentiating $\cA$, we get 
$$
D\cA(\vp)[v]= \int_{S^{N-1}} \cH(\vp)(\s)v(\s) \, d\s,
$$
where $ \cH(\vp)(\s)=H_{S^{N-1}_\vp }( \vp(\s)\s)$ is the mean curvature operator of $\de B^N_\vp=S^{N-1}_\phi$ and is given by
\be \label{eq:MC-operator-sph}
 \cH(\vp) =- \div\frac{ \vp^{N-2} \n \vp}{\sqrt{ \vp^{2}+  |\n \vp|^2}} +(N-2)\frac{\vp^{N-3} }{ \sqrt{ \vp^{2}+  |\n \vp |^2}}+\frac{\vp^{N-1}}{( \vp^{2} +  |\n \vp |^2)^{-1/2}} .
\ee
\subsubsection{From 2D near-balls  to 3D near-cylinders}\label{sp:from-balls-to-cyl}
We define the perturbed cylinder
\be
C_\vp:=  \left\{( r,\th,t)\in
 \R_+\times S^{1}\times \mathbb{R} \,:\, r<\vp(\th)  \right\}= B_\vp^2\times \R.
\ee
We then consider the lattice of cylinders
$$
  \cC_{\vp}^\e:=\bigcup_{p \in \scrL^M} \left(C_\vp+\frac{p}{|\e|}\right)= \bigcup_{p \in \scrL^M} \left(C_\vp+\left(\frac{p}{|\e|},0\right)\right) ,
$$
where  $\vp:   S \to (0,\infty)$ is an even and $\scrL^M$ is the $M$-dimensional Bravais lattice, with $M\leq 2$. \\
We note that   the integral  of the  Yukawa  interaction $G_\k$ in the $t$ direction gives $2K_0$ which is   the $2$-dimensional screened Coulomb potential.  Our problem is therefore reduced to find  spherical lattices of  equilibrium pattern   in   $\R^2$. Indeed, for a point   $Z_0=(\vp(\th)\th,t)\in \de C_\vp$, the Yukawa interaction takes the form
\begin{align*}
\int_{ \cC_{\vp}^\e}G_\k(|Z_0-Z|) dZ&= \sum_{p\in \scrL^M }\int_{ C_\vp}   G_\k(|Z_0-Z-\frac{p}{|\e|} |)dZ\\
&=\int_{ \cC_{\vp}^\e}   G_\k(|Z_0-Z|)dZ+ \sum_{p\in \scrL ^M} \int_{ C_{\vp}}   G_\k(|Z_0-Z-\frac{p}{|\e|} |)dZ\\
&= \int_{B_\vp^2}   \int_{ \R} G_\k(|(\vp(\th)\th,t) -(y, s) -\frac{p}{|\e|}|) ds   dy  \\
&\quad+ \sum_{p\in \scrL^M }  \int_{B_\vp^2}   \int_{ \R} G_\k(|(\vp(\th)\th,t) -(y, s) |) ds   dy     .
\end{align*} 
By \eqref{eq:int-cos-G-k-sqrt},  
$$
 \int_{\R}   G_\k\left((r^2+\d^2)^{1/2} \right)dr= 2K_0( \k \d ),
$$
where as before $K_0$ is the modified Bessel function of the second kind of order $0$.
We then deduce that 
\begin{align*}
\int_{ \cC_{\vp}^\e}G_\k(|Z_0-Z|) dZ 
&=2  \int_{B_\vp^2}    K_0(\k |(\vp(\th) \th)- y|)  dy  \\
&+2\sum_{p\in \scrL_* ^M}    \int_{B_\vp^2}   K_0\left(\k  |(\vp(\th) \th)-y -\frac{p}{|\e|}  | \right)  dy.
\end{align*} 

Since the mean curvature $H_{C_\vp}$ does not depend on $t$, from the computations above, $\cH_{\cC^\e_\vp}$ does not depend on $t$.  We are therefore reduce to find near-ball lattices $B_\vp^2+\frac{1}{|\e|}\scrL^M$ in $\R^2$ with Yukawa potential  (or screened Coulomb potential) $2K_0$. We will therefore treat simultaneously the 2D and 3D cases. For that, we 
 define 
\be \label{eq:def-GkN}
G_{\k,N}(r):=\left\{\begin{array}{cc}
2K_0(\k r)&\qquad\textrm{ for $N=2$}\vspace{3mm}\\
G_\k(r)&\qquad\textrm{ for $N=3$}. 
\end{array}
\right.
\ee

\subsection{  Lattices of perturbed balls   in 2D and 3D }\label{ss:cyl-lattice}

For $|\e|$ small and $\phi\in \cD $, we then define 
$$
\cB_\phi^{\e} := \bigcup_{p \in \scrL^M} \Bigl( B_\phi^N  + \frac{p}{|\e|}\Bigr)=B_{\phi}^N+\frac{1}{|\e|} \scrL^M.
$$
We will consider in this section the Yukawa potential $G_{\k,N}$  defined in \eqref{eq:def-GkN}. 
For $Z_0=\th \phi(\th) \in \de  B_\phi^N$, we have 
\be\label{eq:cH-eq-latt-cyl-sph}
\cH_{\cB_\phi^{\e} } (Z_0)= H_{\de B^N_\phi}(Z_0)+ \g \int_{\cB_\phi^{\e}}G_{\k,N}(|Z_0-Z|) dZ.
\ee

The Yukawa interaction takes the form
\begin{align*}
\int_{\cB_\phi^{\e}}G_{\k,N}(|Z_0-Z|) dZ&= \sum_{p\in \scrL ^M}\int_{ B_\phi^N}   G_{\k,N}(|Z_0-Z-\frac{p}{|\e|} |)dZ\\
&=\int_{ B_\phi^N}   G_{\k,N}(|Z_0-Z|)dZ+ \sum_{p\in \scrL_* ^M} \int_{ B_\phi^N}   G_{\k,N}(|Z_0-Z-\frac{p}{|\e|} |)dZ.
\end{align*} 
We define the following maps   
\be
\cD\to  C^{0,\alpha}(  S^{N-1}),\qquad \cF_{0}(\phi)(\th):=  \int_{ B_\phi^N}   G_{\k,N}(|Z_0-Z|)dZ,
\ee
\be
\R\setminus\{0\}\times \cD\to  C^{0,\alpha}( S^{N-1}), \qquad \cF_{p}(\e, \phi)(\th):=  \int_{ B_\phi^N}   G_{\k,N}(|Z_0-Z-\frac{p}{|\e|} |)dZ
\ee
and    
\be
\cF: \R\setminus\{0\}\times \cD\to  C^{0,\alpha}(  S), \qquad \cF (\e, \phi)(\th):=\sum_{p\in \scrL_* ^M} \int_{ B_\phi^N}   G_{\k,N}(|Z_0-Z-\frac{p}{|\e|} |)dZ.  
\ee
We observe  that for any neighborhood $\cU$ of $1$ in  $\cD$, there exist  constants $C, \e_0>0$ (depending only on $\k$ and $\scrL^M$)  such that for every $|\e|\in (0,\e_0)$ and any $ p\in \scrL_*^M$
$$
|\cF_{p}(\e, \phi)(\th)  |  \leq C |p|^ {-1/2}  |\e|^ {1/2} \exp\left(\frac{- \k |p|}{2|\e|} \right)  ,\qquad\textrm{ for every $\vp\in \cU$ and $\th\in  S^{N-1}$}
$$
and 
\be\label{es:-cFp-eps}
\|\cF(\e, \phi)  \|_{C^{0,\a}(  S^{N-1})}  \leq  \exp(-c |\e|^{-1})  ,\qquad\textrm{ for every $\vp\in \cU$ and $\th\in  S^{N-1}$,}
\ee
for some $c>0$ depending only on $\a,N,\k$ and $\scrL^M$.\\
It follows that   $\cF$ extends smoothly in a neighborhood of $(0,1)$ in $\R\times \cD$.  Now we define the mean curvature operator   $\cH: \cD\to  C^{0,\alpha}(  S^{N-1})$ (computed in \eqref{eq:MC-operator-sph}) by 
\be
\cH(\phi)(\th)=H_{\de B_\phi^N}(\phi(\th)).
\ee
We define $\cQ: \R\times   \cD\to C^{0,\a}(  S^{N-1})$ by 
\be \label{eq:def-cG-Operator-sph}
\cQ(\e,\phi):=\cH(\phi)+\g \cF_{0}(\phi)+\g  \cF(\e, \phi)
\ee
and we note that from \eqref{eq:cH-eq-latt-cyl-sph},
$$
\cH_{\cB_\phi^{\e} }(\th\phi(\th))=\cQ(\e,\phi)(\th).
$$
Our aim is to apply the implicit function theorem   to solve the equation
\be \label{eq:equation-to solve-sphere-lattice}
\cQ(\e,\phi)  -\cQ(0,1)=0.
\ee
Let us then study the linearization of $D_\phi\cQ(0,1)$. 
\begin{Lemma}\label{Regularity -cF-0-lattice-sph-cyl}
The map $\cQ: \cD\to C^{0,\a}(  S^{N-1})$ is smooth. Moreover for every $\l>0$  and $w\in C^{2,\a}(  S^{N-1})$, we have 
\begin{align}
D_\phi \cQ(0,1)[w](\th )& =-\D_{S^{N-1}} w(\th)-(N-1)w(\th) - \g   \int_{S^{N-1}}(w(\th)-w( \s))  G_{\k,N}\left( |\th -\s| \right)  d\s \nonumber  \\
  &\quad + \g w(\th)  \int_{S^{N-1}}(1-\th\cdot \s )  G_{\k,N}\left( |\th -\s| \right)   d\s. \label{eq:express-lin-sph-latt}
\end{align}
\end{Lemma}
  The proof of this result is not immediate (at least for $N=2$). We therefore postpone its proof  to Section \ref{ss:reg-cF_0-latt} below. 
\begin{Lemma}\label{eq:def-linearized-eigen-latt-Del}
We define the linear operator $$L:=D_u\cQ(0,1):C^{2,\a}(  S^{N-1}) \to C^{0,\a}(  S^{N-1}). $$ Then the  spherical harmonics $Y_k^i$
 are  the   eigenfunctions of $L$ and moreover
$$ 
 L(Y_k^i(\th))=\s_{\g}(k) Y_k^i(\th),
$$
 with eigenvalues 
\begin{align}
\s_{\g}(k)
  &= \l_k- \l_1 -  \g ( \mu_k-  {\mu}_1)     ,
\end{align}
where $\l_k=k(k+N-2)$   and, for $i=1,\dots, n_k$,
 $$
 \mu_k =\frac{1}{2}\int_{S^{N-1}\times S^{N-1}}(Y_k^i(\th)-Y_k^i(\s))^2 G_{\k,N}( |\th-\s|) d\th d\s.
$$
%
\end{Lemma}

\begin{proof}
Using Fourier decomposition in spherical harmonics, see Section \ref{ss:prel-not}, we easily deduce that 
\begin{align*}
\s_{\g}(k)
  = \l_k-(N-1)   +\g  \int_{S^{N-1}}(1-\th\cdot \s ) G_{\k,N}\left(|\th -\s|  \right)  d\s \\
  - \frac{\g}{2}  \int_{S^{N-1}\times S^{N-1}}(Y_k(\th)-Y_k(\s))^2  G_{\k,N}\left(    |\th-\s |  \right)  d\th d\s,
\end{align*}
for every $Y_k\in \cS_k$.
We have that 
\begin{align*}
  \int_{S^{N-1}}(1-\th\cdot \s ) G_{\k,N}\left(|\th -\s|  \right)  d\s&= \frac{1}{2}\sum_{i=1}^N \int_{S^{N-1}}(\th_i-\s_i )^2 G_{\k,N}\left(|\th -\s|  \right)  d\s\\
    &=\frac{ |B^N|}{2}\sum_{i=1}^N   \int_{S^{N-1}}(Y^i_1(\th)-Y^i_1(\s) )^2 G_{\k,N}\left(|\th -\s|  \right)  d\s\\
  &=\frac{N |B^N|}{2}  \int_{S^{N-1}}(Y^i_1(\th)-Y^i_1(\s) )^2 G_{\k,N}\left(|\th -\s|  \right)  d\s.
  \end{align*}
  Since the  integral on the  left hand side above does not depend on $\th$, integrating the above equality  and using that $|S^{N-1}|=N|B^N|$, we get 
\begin{align*}
    \int_{S^{N-1}\times S^{N-1} }&(1-\th\cdot \s ) G_{\k,N}\left(|\th -\s|  \right)  d\s d\th\\
    &= \frac{ 1}{2}    \int_{S^{N-1}\times S^{N-1} }(Y^i_1(\th)-Y^i_1(\s) )^2 G_{\k,N}\left(|\th -\s|  \right)  d\s d\th\\
    &=  \mu_1.
 \end{align*}

\QED
\end{proof}

The following result  will be useful in order to study the invertibility of the linearized operator on the space of even function.
\begin{Lemma}\label{eq:Lambda-star-latt-sph}
We let     
\be \label{eq:gam-star-lattice-sph-cyl}
\g_N:=\min \left( \frac{\l_1}{  {\mu}_1}, \inf_{k\geq 2}\frac{\l_k-\l_1}{\mu_k-\mu_1} \right).
\ee
Then    for every $\g\in (0,\g_N)$ 
 \be \label{eq:sk-neq-0-latt-sph}
 \s_{\g}(k)> 0\qquad \textrm{ for all $k\geq 2$},  \qquad\textrm{ }\qquad    \s_{\g}(0)<0
 \ee
 and 
 \be\label{eq:asymp-sigk-latt-sph}
 \lim_{k\to \infty}  \frac{ \s_{ \g}(k)}{k^2}= 1.
 \ee
\end{Lemma}
\begin{proof}
We observe that, for $N=3$, 
$$
\mu_k\leq \mu_k^0:= \frac{1}{2}\int_{S^{N-1}\times S^{N-1}}(Y_k^i(\th)-Y_k^i(\s))^2 G_{0,N}( |\th-\s|) d\th d\s,
$$
and $\mu_k^0$ is the eigenvalues of the well known \textit {hypersingular Riesz operator
on the sphere} (\cite[Section 7]{FFMMM}). In this case,   we have that $\lim_{k\to\infty}\frac{\l_k}{\mu_k^0}=\infty$, see e.g.  \cite[Section 7]{FFMMM}.\\

 For $N=2$, we have       $Y_k^1(\cos(x),\sin(x))=\frac{\cos(k x)}{\pi^{1/2}}$ and $Y_k^2(\cos(x),\sin(x))=\frac{\cos(k x)}{\pi^{1/2}}$. Moreover $K_0(\k r)\leq C  |\log(\k r)|$.  Therefore, we can estimate 
\begin{align*}
\mu_k^0&\leq C \int_{0}^{2\pi}  (\cos(k x)-\cos(k y))^2|\log(|\k \sqrt{1-\cos(x-y)|})|dxdy\\
%
%
%
&\leq C \int_{0}^{\pi} \int_{0}^{\pi} |\log(\k \sqrt{2}|\sin(x-y)|)|dxdy,
\end{align*}
with $C$ independent on $k$. 
Since $\mu_i\neq\mu_j$ for $i\neq j,$ we
  can therefore  define 
$$
\g_N:=\min \left( \frac{\l_1}{\mu_1}, \inf_{k\geq 2}\frac{\l_k-\l_1}{\mu_k-\mu_1} \right).
$$
It follows that for every $\g\in (0, \g_N)$, we have $\s_\g(0)<0$ and $\s_\g(k)>0$ for every $k\geq 2$. 
\QED
\end{proof}
We then have the following result. 
\begin{Lemma}\label{lem:rang-L-star-latt-sph}
Let $\g\in (0, \g_N)$. 
Then the linear map $$\cL_{\textbf{e}}:=L\big|_{C^{2,\a}_{\textbf{e}}(S^{N-1})}: C^{2,\a}_{\textbf{e}}(S^{N-1}) \to C^{0,\a}_{\textbf{e}}(S^{N-1}) $$ is an isomorphism.
\end{Lemma}
\begin{proof}
By Lemma \ref{eq:Lambda-star-latt-sph} the elements of the Kernel of $L$ correspond with the eigenvalue $\s_\g(1)$ which is the eigenvalues corresponding to   the first non-constant  spherical harmonics $Y_1^1$ and $Y_1^2$ which  are odd. This implies that the kernel of $\cL_{\textbf{e}}$ is  $ \{0\}$.   By Lemma \ref{eq:Lambda-star-latt-sph}, for every $f\in Y\in L^2_{\textbf{e}}(S^{N-1})$, there exists a unique $w\in H^2_{\textbf{e}}(S^{N-1})$ such that $L(w)=f$.  Since $2\leq N\leq 3$,  by Morrey's embedding theorem, $w\in C^{0,\a}(S^{N-1})$ for every $\a\in (0,1)$.\\
We define
$$
F(\th)= \int_{S^{N-1}}(w(\th)-w(\s)) G_{\k,N}( |\th -\s| )  d\s
$$
so that 
\begin{align}\label{eq:Lw-eq-F-f}
L(w) =-\D_{S^{N-1}} w(\th) -  \l_1 w(\th ) +2\g  \mu_1 w(\th ) - 2 \g   F(\th)= f    .
\end{align}

Then for every $\th\in S^{N-1}$, 
\be \label{eq:F-in-Holder}
|F(\th)|\leq C \|w\|_{C^{0,\a}(S^{N-1})} \int_{S}|\th -\s|^\a G_{\k,N}( |\th -\s| )  d\s\leq   C \|w\|_{C^{0,\a}(S^{N-1})} .
\ee
We fix $e \in S^{N-1}$ and $R_\th$ be given by Example \ref{example-rotations}.  
Thanks to \eqref{eq:Rota-isom},   a change of variable gives
$$
F(\th)=  \int_{S^{N-1}}(w(\th)-w(R_\th \s)) G_{\k,N}( |e -\s| )  d\s.
$$
By \eqref{eq:Rota-isom} and \eqref{eq:R-Lipschitz},  for all $\th_1,\th_2\in S_e$ and $\s\in S^{N-1}$, we have 
$$
|(w(\th_1)-w(R_{\th_1} \s))-(w(\th_2)-w(R_{\th_2} \s)) |\leq C  \|w\|_{C^{0,\a}(S)} \int_{S} \min(  |e -\s|^\a, |\th_1-\th_2|^\a ).
$$
We assume that $|\th_1-\th_2 |\leq 1/4$.   Then by \eqref{eq:decKnuInf0}, we have 
\begin{align*}
&|F(\th_1)-F(\th_2)|\leq C \|w\|_{C^{0,\a}(S^{N-1})}\int_{S^{N-1}}  \min(  |e -\s|^\a, |\th_1-\th_2|^\a )    |e -\s|^{-1} d\s\\
&\leq  C  \|w\|_{C^{0,\a}(S^{N-1})}  \int_{ |e -\s| <|\th_1-\th_2|}  |e -\s|^{-1+\a}   d\s\\
&\quad+ C  \|w\|_{C^{0,\a}(S^{N-1})} |\th_1-\th_2|^\a  \int_{1>  |e -\s|\geq |\th_1-\th_2|} |e -\s|^{-1}  d\s\\
&\leq  C   \|w\|_{C^{0,\a}(S^{N-1})}  |\th_1-\th_2|^\a(-\log(|\th_1-\th_2 |)).
\end{align*}
From this and \eqref{eq:F-in-Holder}, 
it then follows that $F\in C^{0,\b}(S^{N-1})$ for every $\b\in (0,1)$. Now from \eqref{eq:Lw-eq-F-f}, we can apply standard   elliptic regularity estimates to  deduce that $w\in C^{2,\b}_{\textbf{e}}(S^{N-1})$.
\QED
\end{proof}
\subsubsection{Proof  of Theorem   \ref{th:latt-cyl} and \ref{th:latt-sph} (completed)}
We  recall that    the  function $$\cQ: \R \times   C^{2,\a}_{\textbf{e}}(S^{N-1})\cap \cD \to  C^{0,\a}_{\textbf{e}}(S^{N-1}) $$ is well-defined and   smooth in a neighborhood of $(0, 1)\in \R\times  C^{2,\a}_{\textbf{e}}(S^{N-1})\cap \cD$. Moreover, by Lemma \ref{lem:rang-L-star-latt-sph}, the linear map  $D_\vp  \cQ(0, 1)=\cL_e : X \to Y $ is an isomorphism. 
 We can therefore apply the implicit function theorem  to obtain a  constant $\e_0 >0$ and a smooth curve
\be
(-\e_0,\e_0)  \longrightarrow  X ,\qquad \e\mapsto  \o_{\e}
\ee
such that  $\|\o_\e\|_{C^{2,\a}(S^{N-1})}\to 0$ as $\e\to 0$  and 
\be \label{eq:latt-solve}
\cQ(\e, \o_{\e})=\cQ(0, 1) \qquad \textrm{  on $ C^{0,\a}_{\textbf{e}}(S^{N-1})$.}
\ee
  Moreover it is easy to see from \eqref{es:-cFp-eps} and \eqref{eq:latt-solve} that 
\be \label{eq:est-o-eps-latt-sph-cyl}
\|\o_\e\|_{C^{2,\a}(S^{N-1})}\leq  \|\cL_e^{-1} (\cF(\e,1+\o_\e))\|_{ C^{2,\a}(S^{N-1})} \leq e^{-\frac{c}{|\e|} },
\ee
for some positive constants $c$.
\QED
\subsection{ Remarks on the perturbations of near-cylinders and near-balls}\label{s:Remarks}
The aim of this section is to show that the perturbation $\vp_\e$ and $\phi_\e$ from  Theorem   \ref{th:latt-cyl} and \ref{th:latt-sph} are not constant in the respective  lower dimensional lattice $M=1$ (for cylinder) and $1\leq M\leq 2$ (for the balls).  This is equivalent to show that $\cC_{1}^\e$ and $\cB_{1}^\e$ are not equilibrium patterns,  for small $\e$.
\begin{Proposition}
Let $\vp_\e$ and $\phi_\e$ be given by    Theorem \ref{th:latt-cyl} and  \ref{th:latt-sph}, respectively,  and let $1\leq M\leq N$, with $N=2,3$. Then 
$$
\vp_\e(\th)=1-  \cL_e^{-1}(U_\e)(\th)+ \e^{5/2}  O\left( \sum_{p\in\scrL_*^M} {|p|^{-3/2}} e^{ - \k\frac{|p|}{\e}} \right)
$$ 
and 
 $$
\phi_\e(\th)=1-  \cL_{\textbf{e}}^{-1}(V_\e)(\th)+ \e^2    O\left(\sum_{p\in\scrL_*^M} |p|^{-1}e^{-\k\frac{|p|}{\e} }\right).
$$
Here   $(\cL_{\textbf{e}})^{-1}: C^{0,\a}_{\textbf{e}}(S^{N-1})\to C^{2,\a}_{\textbf{e}}(S^{N-1})$ denotes the inverse of $\cL_{\textbf{e}}$ (see \ref{lem:rang-L-star-latt-sph}).  Moreover 
$$
U_\e(\th)=  \sum_{p\in\scrL_*^M} \cosh\left(\k \frac{\th\cdot p }{|p|}\right)  \xi_\e(|p|),\qquad V_\e(\th)=  \sum_{p\in\scrL_*^M}  \cosh\left(\k \frac{\th\cdot p }{|p|}\right)  \z_\e(|p|),
$$
with 
$$
\xi_\e(\ell)= c_2\frac{\sqrt{\pi}}{\sqrt{2}} \frac{\e^{3/2}}{\ell^{5/2}} e^{ -\k  \frac{\ell}{\e}},\qquad \xi_\e(\ell)= c_3  \frac{\e}{\ell} e^{ -\k  \frac{\ell}{\e}},
$$
where 
 $$
c_N:=    \int_{B^N}e^{-\k y\cdot e_1}\,dy.
 $$
Moreover if $M=1$ then $\vp_\e$ is not constant and if $M\leq 2$ then $\phi_\e$ is not constant.
\end{Proposition}
\proof
  We  consider  the function $f_\e: S^1\to \R$, given by 
$$
f_\e (\th)=\int_{\cC_1^\e\setminus C_1}G_\k(|\th-y-\frac{p}{\e}|) dy=2 \sum_{p\in \scrL^M_*} \int_{B_1^2}K_0(\k |\th-z-\frac{p}{\e}|) dz.
$$
For $\e>0$,  we write
 $$
 |\th-z-\frac{p}{\e}|=\frac{|p|}{\e}(1+m_\e)^{1/2}=\frac{|p|}{\e}\left(1+\frac{1}{2}m_\e+ O(\e^2/|p|)\right),
 $$
 where 
 $$
 m_\e=\frac{\e^2|y-\th|+2\e(y-\th)\cdot p}{|p |^2}. 
 $$
Therefore
$$
|\th-z-\frac{p}{\e}|=\frac{|p|}{\e}\left(1+ \frac{\e(y-\th)\cdot p }{|p|^2}+ O(\e^2/|p|)\right).
$$
By \eqref{eq:decKnuInf}, 
$$
K_0(r)\sim \frac{\sqrt{\pi }}{\sqrt{2}} r^{-1/2}e^{-r} \qquad \textrm{  as
$r\to +\infty$} .
$$
We then deduce that, as $\e\to 0$, 
\begin{align*}
\frac{2\sqrt{2 }}{\sqrt{\pi }} K_0( |\th-z-\frac{p}{\e}|) &=\frac{|p|^{-1/2}}{\e^{-1/2}}\left(1- \frac{\e(y-\th)\cdot p }{2|p|^2}+ O(\e^2/|p|)\right)e^{ - \k \frac{|p|}{\e}} e^{-\k \frac{(y-\th)\cdot p }{|p|} } e^{O(\e)}\\
&=\left( \frac{\e^{3/2}}{|p|^{5/2}} e^{ - \k \frac{|p|}{\e}} e^{-\k \frac{(y-\th)\cdot p }{|p|} } + O( {\e^{5/2}}{|p|^{-3/2}} e^{ -\k \frac{|p|}{\e}})\right)e^{O(\e)} \\
&=\left( \frac{\e^{3/2}}{|p|^{5/2}} e^{ - \k\frac{|p|}{\e}} e^{-\k \frac{y\cdot p }{|p|} } e^{\k\frac{\th\cdot p }{|p|} } + O( {\e^{5/2}}{|p|^{-3/2}} e^{ - \k\frac{|p|}{\e}}) \right) .
\end{align*}
We define 
 $$
c_2:=  \int_{B^2}e^{-\k \frac{y\cdot p }{|p|} }\,dy= \int_{B^2}e^{-\k y\cdot e_1}\,dy
 $$
 and 
\be \label{eq:def-xi}
\xi_\e(\ell)= c_2\frac{\sqrt{\pi}}{\sqrt{2}} \frac{\e^{3/2}}{\ell^{5/2}} e^{ -\k  \frac{\ell}{\e}}  .
\ee
To see that $f_\e$ is not constant, it suffices to check that the map
$$
\th\mapsto U_\e(\th):= \sum_{p\in\scrL_*^M} e^{-\k \frac{\th\cdot p }{|p|} } \xi_\e(|p|)
$$
is not constant.
Since  $\scrL_*^M=-\scrL_*^M$, we deduce that
$$ 
U_\e(\th):= \sum_{p\in\scrL_*^M} \cosh \left(\k \frac{\th\cdot p }{|p|} \right) \xi_\e(|p|).
$$
We then conclude that 
\be \label{eq:Taylor-of-f-eps}
f_\e(\th)=  \int_{\cC_1^\e\setminus C_1}G_\k(|\th-y-\frac{p}{\e}|) dy=    U_\e(\th)+\e^{5/2}  O\left( \sum_{p\in\scrL_*^M} {|p|^{-3/2}} e^{ - \k\frac{|p|}{\e}} \right).
\ee
Provided $M=1$, we may assume that $\scrL^1=\Z\times \{0\}\subset\R^2$ the subspace of $\R^2= \textrm{span}\{e_1,e_2\}$ with $e_1=(1,0)$ and $e_2=(0,1)$. Therefore  we have  that $e_2\cdot p=0$ and thus 
$$
U_\e(e_2)=\sum_{p\in\Z_*}  \xi_\e(|p|).
$$
However
$$
U_\e(e_1)= \cosh(\k)\sum_{p\in\Z_*}   \xi_\e(|p|).
$$
This implies  that $U_\e(e_1)> U_\e(e_2)$.  That is $\cH_{ \scrC_1^\e}(e_1)\neq \cH_{ \scrC_1^\e}(e_2)$, for $\e>0$ small.\\

We now turn to the sphere lattices and provide Taylor expansion of  $\phi_\e$ in Theorem \ref{th:latt-sph}, and prove that it is not constant for $M\leq 2$. We adopt the same strategy as a above. We  consider  the function $F_\e: S^2\to \R$, given by 
$$
F_\e (\th)=  \int_{\cB_1^\e\setminus B_1^2}G_\k(|\th-y |) dy= \sum_{p\in \scrL^M_*} \int_{B_1^2}G_\k( |\th-z-\frac{p}{\e}|) dz.
$$
Using the Taylor expansions as above, we get  
\begin{align*}
 G_\k( |\th-z-\frac{p}{\e}|)&=\frac{\e}{|p|}\left(1- \frac{\e(y-\th)\cdot p }{|p|^2}+ O(\e^2/|p|)\right) e^{-\k\frac{|p|}{\e}\left(1+ \frac{\e(y-\th)\cdot p }{|p|^2}+ O(\e^2/|p|)\right) }\\
 &= \frac{\e}{|p|}e^{-\k\frac{|p|}{\e} } e^{O(\e)} e^{-\k (y-\th)\cdot \frac{p}{|p|}} (1+O(\e^2/|p|))\\
 &= \frac{\e}{|p|} e^{-\k\frac{|p|}{\e} } e^{-\k (y-\th)\cdot \frac{p}{|p|}}+ O\left(\frac{ \e^2}{|p|} e^{-\k\frac{|p|}{\e} }\right).
\end{align*}
We define 
$$
c_2:= \int_{B^3} e^{-\k y\cdot \frac{p}{|p|}}dy = \int_{B^3} e^{-\k y\cdot e_1}dy
$$
and    
\be \label{eq:def-z}
\z_\e(\ell):= c_1 \frac{\e}{\ell} e^{ -\k \frac{\ell}{\e}}  .
\ee
Letting
$$
V_\e(\th):= \sum_{p\in\scrL_*} e^{\frac{\th\cdot p }{|p|} \k} \z_\e(|p|),
$$
it  follows that 
\be \label{eq:Taylor-of-F-eps}
F_\e(\th)=\sum_{p\in\scrL_*^M} G_\k( |\th-z-\frac{p}{\e}|)=V_\e(\th)+\e^2    O\left(\sum_{p\in\scrL_*^M} |p|^{-1}e^{-\k\frac{|p|}{\e} }\right).
\ee
Since $\scrL_*^M=-\scrL_*^M$, we deduce that
$$ 
V_\e(\th):= \sum_{p\in\scrL_*^M} \cosh \left(\k\frac{\th\cdot p }{|p|} \right) \z_\e(|p|)=\sum_{\ell\in \N_*} \sum_{p\in E_\ell}\cosh(\k\th\cdot \frac{p}{\ell})\z_\e(\ell),
$$
where $E_\ell:=\{p\in \scrL_*^M\,:|p|=\ell\}$. 
We are therefore reduced to prove that $V_\e$ is not constant.  Now if
  $M\leq 2$,  then   $e_3\cdot p=0$ and thus 
$$
V_\e(e_3)=\sum_{\ell \in \N_*}| E_\ell|\z_\e(\ell).
$$
We have 
$$ 
V_\e(e_1):= \sum_{\ell\in \N_*} \sum_{p\in E_\ell}\cosh(\k e_1\cdot \frac{p}{\ell})\z_\e(\ell).
$$
It is clear that for some  $\ell\in\N_*$, we must have $\sum_{p\in E_\ell}\cosh(\k e_1\cdot \frac{p}{\ell})>| E_\ell|$, since $e_1 $ cannot be perpendicular to both $\textbf{a}_1$  and   $\textbf{a}_2$.

It follows from the construction, via the Implicit Function Theorem, of near-cylinder and the near- sphere lattices that, $\vp_\e$  in Theorem \ref{th:latt-cyl} has the following Taylor expansion
$$
\vp_\e=1-  \cL_{\textbf{e}}^{-1}(U_\e)+ \e^{5/2}  O\left( \sum_{p\in\scrL_*} {|p|^{-3/2}} e^{ - \k\frac{|p|}{\e}} \right).
$$ 
and   $\phi_\e$  in Theorem \ref{th:latt-sph}  Taylor expands as 
 $$
\phi_\e =1-  \cL_{\textbf{e}}^{-1}(V_\e)+ \e^2    O\left(\sum_{p\in\scrL_*} |p|^{-1}e^{-\k\frac{|p|}{\e} }\right).
$$
Here $\cL_{\textbf{e}}$ is given by    Lemma \ref{lem:rang-L-star-latt-sph}  (with the corresponding dimensions $N=2,3$) and 
$$
U_\e(\th)=  \sum_{p\in\scrL_*^M} \cosh\left(\k \frac{\th\cdot p }{|p|}\right)  \xi_\e(|p|),\qquad V_\e(\th)=  \sum_{p\in\scrL_*^M}  \cosh\left(\k \frac{\th\cdot p }{|p|}\right)  \z_\e(|p|),
$$
with $\xi_\e$ and $\z$ defined in \eqref{eq:def-xi} and \eqref{eq:def-z} respectively.  The above expressions of $\vp_\e$ and $\phi_\e$ follow from 
\eqref{eq:Taylor-of-f-eps} and  \eqref{eq:Taylor-of-F-eps} without any restriction on $M$, the dimension of the lattices. 
\QED
%
%

%
\section{Appendix: Proof of Lemma \ref{Regularity -cF-0-lattice-sph-cyl} }\label{ss:reg-cF_0-latt}
We  prove  the regularity of the map $\cQ: \cD\to C^{0,\a}(S^{N-1})$. We note that  $\cH$ and $\cF$ are smooth in $\cD$ and $(-\e_0,\e_0)\times \cD$, for small $\e_0$, respectively.  We  will only consider $\cF_{0,N}$ in the following. Recalling the notations in  Section \ref{ss:prel-not}, we have  the following result.
\begin{Lemma}\label{lem:smooth-cF-0-N}
The map $\cF_{0,N}: \cD\to C^{0,\a}(  S^{N-1})$ is smooth. Moreover for every $\phi\in \cD$  and $w\in C^{2,\a}(  S^{N-1})$, 
 \begin{align}\label{eq:deriv-cF_0-full}
 D \cF_{0,N}( \phi)[w](\th)  =\int_{S^{N-1}}{(w(\s)\s-w(\th)\th)\cdot \nu_\phi(\s)}G_{\k,N}(|F_\phi(\th)-F_\phi(\s) | )\G_\phi(\s)\, d\s  ,
 \end{align}
 where $F_\phi:=F_\phi(1, \cdot)$.
\end{Lemma}

We first  make a formal   proof of  \eqref{eq:deriv-cF_0-full}. Then we  prove regularity estimates for the map   $\phi\mapsto D\cF_{0,N}(\phi)[w]$, for every fixed $w$. This will end the proof of the lemma.\\
For every $\d>0$, we define
\be
\cD_\d:=\{\phi\in C^{2,\a}(S^{N-1})\,:\, \phi>\d\}, 
\ee
Let $\phi\in \cD_\d$ and $w\in C^{2,\a}(S^{N-1})$ such that $\|w\|_{C^{2,\a}(S^{N-1})}< \d/2$.  We recall our definition of  $F_\vp $ in  \eqref{eq:def-F_vp-paramet} and    $S_\phi^{N-1}:=\de B^N_\phi$ for every $\phi \in \cD_\d$.  
 Using polar coordinates, \eqref{eq:normal-S-vp} and \eqref{eq:transformation-rule},  we have 
 \begin{align*}
 D \cF_{0,N}( \phi)[w](\th)& =  \frac{d}{dt}\Big|_{t=0} \int_{S^{N-1}}\int_{0}^{\phi(\s)+tw(\s)}G_{\k,N}(|F_{\phi+t w}(\th)-r\s |)r^{N-1} dr d\s\\
 &= \int_{S^{N-1}} G_{\k,N}(|F_\phi(\th)-\s \phi(\s) |)\phi^{N-1} w(\s) d\s\\
& + \int_{S^{N-1}}\int_{0}^{\phi(\s) }G_{\k,N}'(|F_\phi(\th)-r\s |) w(\th) \th\cdot (\th\phi(\th)-\s r))r^{N-1} dr\\
  & =\int_{S^{N-1}}{\s w(\s)\cdot \nu_\phi(\s)}G_{\k,N}(|F_\phi(\th)-F_\phi(\s) | )\G_\phi(\s)\, d\s \\
  &  - \int_{B_\phi^{N}}\n_y G_{\k,N}(|F_\phi(\th)-y | )\cdot \th w(\th ) \, dy  ,
 \end{align*}
 where $\G_\phi, \nu_\phi$ are defined in Section \ref{ss:prel-not}.
Using the divergence theorem in the last integral, we obtain
 \begin{align*}
& D \cF_{0,N}( \phi)[w](\th)  =\int_{S^{N-1}}{(w(\s)\s-w(\th)\th)\cdot \nu_\phi(\s)}G_{\k,N}(|F_\phi(\th)-F_\phi(\s) | )\G_\phi(\s)\, d\s  ,
 \end{align*}
 which is the expression in \eqref{eq:deriv-cF_0-full}. We will now prove that $\cF_{0,N} $ is smooth by deriving regularity estimates for $ D \cF_{0,N}( \phi)[w] $ computed above.\\
 By  \eqref{eq:normal-S-vp} and \eqref{eq:transformation-rule},  we have
   $$
  (\th-\s)  \cdot \nu_{\phi}(F_\phi(\s)) \, \G_\phi(\s)=  - (\th-\s)  \cdot\n\phi(\s) \phi^{N-2 }(\s)  +(\th-\s)  \cdot\s   \phi^{N-1 }(\s),
  $$
so that
   \begin{align}
 D \cF_{0,N}( \phi)[w](\th) & =\int_{S^{N-1}}(w(\s)-w(\th) )G_{\k,N}(|F_\phi(\th)-F_\phi(\s) | )\phi^{N-1}(\s) \, d\s \nonumber \\
&+{w(\th)}\int_{S^{N-1}}(\th-\s )\cdot \n\phi (\s)G_{\k,N}(|F_\phi(\th)-F_\phi(\s) | ) \phi^{N-2 }(\s) \, d\s \\
&+\frac{w(\th)}{2}\int_{S^{N-1}}|\th-\s |^2 G_{\k,N}(|F_\phi(\th)-F_\phi(\s) | ) \phi^{N-1 }(\s) \, d\s, \nonumber
 \end{align}
  where we have used $2(\th-\s)  \cdot\s=-|\th-\s|^2$.  To prove the regularity of $\phi\mapsto  D \cF_{0,N}( \phi)[w](\th)  $, for a fixed $w\in C^{2,\a}(S^1)$, it is enough to consider the first term. In fact, we will prove a stronger result than needed. \\
  
 For $\b\in (0,1)$ and $\d>0$, we consider the open sets 
  $$
  \cA:=\{\phi\in C^{1,\b}(S^{N-1})\,:\, \phi>0\}
  $$
  and
   $$
  \cA_\d:=\{\phi\in C^{1,\b}(S^{N-1})\,:\, \phi>\d\}.
  $$
We define the function  (the first term  in the expression of $ D \cF_{0,N}( \phi)[w](\th) $)  given by  
   \be\label{eq:def-h-1st-term-cF0}
  \cA_\d \to L^\infty(S^{N-1}),\qquad \phi\mapsto  h(\phi):= \int_{S^{N-1}}(w(\s)-w(\th) )G_{\k,N}(|F_\phi(\th)-F_\phi(\s) | )\phi^{N-1}(\s) \, d\s.  
 \ee 
 Recalling the notations in Section \ref{ss:prel-not}, by  a change of variable,  for every $\th\in S_e$,  we get 
     \begin{align}\label{eq:def-h-deriv-cF0N}
h(\phi)(\th)  =\int_{S^{N-1}}(w(R_\th\s)-w(\th) )G_{\k,N}(|F_\phi(\th)-F_\phi(R_\th\s) | )\phi^{N-1}(\s) \, d\s  .
 \end{align}
Once the following proposition is proved, Lemma \ref{lem:smooth-cF-0-N} follows immediately.
 \begin{Proposition}\label{Propo-reg-h}
 For every $\e\in (0,\b)$, 
the map $h:\cA\subset C^{1,\b}(S^{N-1})\to C^{0,\b-\e}(S^{N-1})$  is of class $C^\infty$. Moreover 
 \begin{align*}
\|D^k h(\phi)\| \leq  c\left(1+ \|{\phi }\|_{C^{1,\b}(S^{N-1})} \right)^{c}    \|w\|_{C^{\b}(S^{N-1})} ,
\end{align*}
for some positive constant $c$ depending only on  $\k,\b$ and $\e$. 
In addition if $N=3$, we can take $\e=0$.  In particular $\cF_{0,N}$ is of class $C^\infty$ in $\cD$.
\end{Proposition}

 We now distinguish the two case $N=2$ and $N=3$. We start with the \\
 
 \noindent 
\textbf{Case $N=2$}.\\
\noindent 
Here, $h$ takes the form (recall \eqref{eq:int-cos-G-k-sqrt}) 
     \begin{align*}
h(\phi)(\th) & =\int_{\R}\int_{S^{1}}(w(R_\th\s)-w(\th) )G_{\k}(|(F_\phi(\th),0)-(F_\phi(R_\th\s) ,0)+t^2 e_3 | )\phi (\s) \, d\s dt. 
 \end{align*}
 We have
 $$
 |F_\phi(\th)-F_\phi( \s)|^2 +t^2=\L_0(\phi,\th, \s)^2+\phi( \s)\phi(\th)|\th-\s|^2+ t^2.
 $$
  Then with the change of variable $\rho=\frac{t}{|\th-\s|}$,   we get 
$$
h( \phi)(\th)  =\int_{\R}\int_{S^{1}}(w(R_\th\s)-w(\th) )\cK(\phi,\th,R_\th\s, \rho)\phi (R_\th\s )  \, d\s d\rho, 
$$
where $\cK:\cA_\d\times S^{ 1}\times S^{ 1}\times\R\to \R$ is given by 
$$
\cK(\phi,\th,\s,\rho ):=  G_{|\th-\s|\k}\left(\L_0(\phi,\th,\s )^2+\phi(\s)\phi(\th) +\rho^2 \right),
$$
with 
$$
\L_0(\phi, \th,\s)=\frac{    \phi(\th)-\phi( \s) }{|\th-\s|}.
$$
See e.g. \cite{Cabre2015B},
there exists a  constant $C>0$ such that for all $\s \in S$,  all $\th,\th_1,\th_2 \in S_e$  and all $\phi_1,\phi_2 \in C^{1,\b}(S^1)$ we have 
\be \label{eq:estL2-th-sig}
 | \L_0(\phi, R_\th \s, \th)|\leq C  \|\phi\|_{C^{1,\b}(S)}\|\phi\|_{C^{1,\b}(S^1)}  .
 \ee 
and 
\begin{align}
| \L_0(\phi, R_{\theta_1} \s,\th_1)&- \L_2(\phi,  R_{\theta_2} \s, \th_2)| \label{eq:estL2-th1-sig1}\\
&\leq  C \|\phi\|_{C^{1,\b}(S)}  |\theta_1-\theta_2|^\b. \nonumber
\end{align}
We follow the arguments as in  \cite{Cabre2015B}, where a more singular kernel was considered. 
\begin{Lemma}\label{lem:est-cK} For every $e\in S^1, k\in \N,\d>0$, there exists a constant   $c=c(\k,k,\d,\a)>0$ such that for $\phi\in\cA_\d$, $\rho\in \R$, $\th, \th_1,\th_2\in S_e$ and $ \s\in S^1$, we have 
$$
\|D^k_\phi\cK(\phi,\th,R_\th\s,\rho) \|\leq c (1+\|\phi\|_{C^{1,\b}(S)})^c \exp(-\k|\s-e|(\d^2+\rho^2))
$$
and
\begin{align*}
 \|D^k_\phi\cK(\phi,\th_1,R_{\th_1}\s,\rho )&- D^k_\phi \cK(\phi,\th_2,R_{\th_2}\s,\rho)\|\nonumber\\
&\leq c (1+\|\phi\|_{C^{1,\b}(S^1)})^c\, |\th_1-\th_2|^\b\, \exp(-\k|\s-e|(\d^2+\rho^2)).
\end{align*}

\end{Lemma}
\proof
We have
$$
\cK(\phi,\th,R_\th\s,\rho ):=  G_{|\s-e|\k}\left(\L_0(\phi,\phi,\th,R_\th\s )^2+\phi(R_\th \s)\phi(\th) +\rho^2 \right).
$$
In view of the arguments in \cite{Cabre2015B}, it suffices to have some estimates related 
$$x\mapsto G_{a}(( x+\rho^2)^{1/2})=\frac{\exp(-a( x+\rho^2)^{1/2})}{( x+\rho^2)^{1/2}}$$ for every $x>\d^2 $, $\rho\in\R$ and $a\in (0,2\k)$. Now elementary computations show that 
\be \label{eq:est-deriv-G-a-latt}
\de^m_x G_{a}(( x+\rho^2)^{1/2})\leq c_{\k,\d,m} \exp(-a( x+\rho^2)^{1/2}) \qquad \textrm{ for  all $x>\d^2 $, $\rho\in \R$, $a\in (0,2\k)$.}
\ee
As in \cite{Cabre2015B}, combining this  estimate together with \eqref{eq:estL2-th-sig}, \eqref{eq:estL2-th1-sig1},  and eventually with the Fa\'{a} de Bruno formula  (see e.g. \cite{FaadeBruno-JW}), we get the estimates in the lemma. 
 
\QED

In the following, for a function $f: S^1 \to \R$, we use the notation
$$
[f; \theta_1,\theta_2]:= f(\theta_1)-f(\theta_2)\qquad \text{ for $\theta_1,\theta_2 \in S^1$,} 
$$
and we note the obvious equality 
\be \label{eq:prod-bracket}
[fg; \theta_1,\theta_2] = [f ;\theta_1,\theta_2]g(\theta_1) +f(\theta_2)[g;\theta_1,\theta_2] \quad \text{for $f,g: S^1 \to \R$, $\theta_1,\theta_2 \in S^1$.}
\ee
 
 The next step is to have estimates of all possible candidates for the derivatives of $h$. This is the purpose of the following 
 \begin{Lemma}\label{lem:est-cand-deriv}  
For $\b\in (0,1)$ and $\d>0$, we let ${\phi } \in \cA_\d$, with $\d>0$.  Let   $k\in \N $.  Let also $W,w\in C^{\b}(S^1)$ and ${\psi}_1,\dots, {\psi}_k \in C^{1,\b}(S^1)$.   Define the functions $\ti{h} : S^1 \to \R$ by 
\begin{align*}
\ti{h}(\th )&=\int_{{S^1}} (w(R_\th\s)-w(\th)) W(R_\th \s)\, D_\phi^k {\cK}(\phi,R_\th \s,\th,\rho)[\psi_1,\dots,\psi_k]     \,  d\s d\rho,
\end{align*}
 
 Then for every $\e\in (0,\b)$ there exists a constant  $c=c(N,\b,k,\d,\e)$ such that 
\begin{align}\label{eq:est-cF1}
\|\ti{h}(\th ) \|_{C^{\b-\e}(S^1)} & \leq  c\left(1+ \|{\phi }\|_{C^{1,\b}(S^1)} \right)^{c}   \prod_{i=1}^k \|{\psi}_i\|_{C^{1, \b}} \|w\|_{C^{\b}(S^1)}   \|W\|_{C^{\b}(S^1)}  .
\end{align}
\end{Lemma}
\proof
Let $\th_1,\th_2 \in S^1_e$ and $\s\in S^1$. We first note that 
\begin{equation}
  \label{easy-est-G}
|W(R_{\th_1} \s)- W(R_{\th_2}\s)| \le C \|W\|_{C^{\b}(S^1)} |R_{\th_1} \s - R_{\th_2} \s|^\b \le C
\|W\|_{C^{\b}(S)}   |\th_1-\th_2|^\b,
\end{equation}
\begin{align}
  \label{eq:estL1-x1-x2}
|(w(R_{\th_1} \s))-w(\th_1))-( w(R_{\th_2}\s)-w(\th_1))| \le C \|w\|_{C^{\b}(S^1)}\min( | \s-e|^\b, |\th_1-\th_2|^\b)
\end{align}
and 
\begin{align}
  \label{eq:estL1-x}
|(w(R_{\th} \s))-w(\th)) | \le C \|w\|_{C^{\b}(S^1)}  | \s-e|^\b .
\end{align}
Using inductively \eqref{eq:prod-bracket} together with Lemma \ref{lem:est-cK}, (\ref{easy-est-G}), \eqref{eq:estL1-x} and \eqref{eq:estL1-x1-x2},  we get the estimate
\begin{align}\label{intermed-est-F_1}
&|[\ti{h};\th_1,\th_2 ]|  \leq c\left(1+ \|\phi \|_{C^{1,\b}(S^1)}^{c} \right)\|w \|_{C^{\b}(S^1)}\|W\|_{C^{\b}(S^1)}  
 \prod_{i=1}^k \|\psi_i\|_{C^{1,\b}(S^1)} \nonumber\\
& \quad \times   \left(     |\th_1-\th_2|^\b\!\! \int_{\R} \int_{S^{1}} | \s-e|^\b \exp(-\k |\s-e|\rho) d\s d\rho\right. \nonumber\\
&\left.\quad+  \int_{\R} \int_{S^{1}}\min( | \s-e|^\b, |\th_1-\th_2|^\b) \exp(-\k |\s-e|\rho) d\s d\rho\right)   \nonumber\\
&=    c\left(1+ \|\phi \|_{C^{1,\b}(S^1)}^{c} \right)\|w \|_{C^{\b}(S^1)}\|W\|_{C^{\b}(S^1)}  
 \prod_{i=1}^k \|\psi_i\|_{C^{1,\b}(S^1)} \times    \nonumber\\
&\left(     |\th_1-\th_2|^\b\!\! \int_{\R} \int_{S^{1}} \frac{\exp(-\k r) }{| \s-e|^{1-\b}}   d\s dr +  \int_{\R} \int_{S^{1}}\frac{\min( | \s-e|, |\th_1-\th_2|) \exp(-\k r)}{ |\s-e|} d\s dr \right).
\end{align}
For all $\th_1,\th_2\in S$, we then have  
\begin{align*}
\int_{S^{1}}\frac{\min( | \s-e|^\b, |\th_1-\th_2|^\b)}{ |\s-e|}   d\s 
&\leq  \int_{ |e-\s| \leq |\th_1-\th_2|}    | \s-e|^{\b-1} d\s\\
&\quad + |\th_1-\th_2|^\b  \int_{  |\th_1-\th_2|\leq |e-\s| \leq 2}|\s-\th|^{-1}  d\s\\
&\leq c |\th_1-\th_2|^\b+ c  |\th_1-\th_2|^\b |\log(|\th_1-\th_2|  )| .
\end{align*}
We thus deduce from (\ref{intermed-est-F_1}) that for all $\e\in (0,1)$  
\begin{align*}
|[\ti{h};\th_1,\th_2 ]| &  \leq c |\th_1-\th_2|^{\b-\e} \left(1+ \|\phi \|_{C^{1,\b}(S^1)}^{c} \right)\|w \|_{C^{\b}(S^1)}  
\|W\|_{C^{\b}(S^1)}  \prod_{i=1}^k \|\psi_i\|_{C^{1,\b}(S^1)}.
\end{align*}
\QED
\subsubsection{Completion of the proof of Proposition \ref{Propo-reg-h} for $N=2$.}
%
We consider 
\begin{align*}
\ov{h}(\th)&=\sum_{i=1}^k\int_\R \int_{S^1}(w(R_\th\s)-w(\th)) \phi_i(R_\th\s) D^k_\phi \cK(\phi,\th,R_\th\s,\rho)[\psi_1,\dots,\psi_k]d\s d\rho\\
&+ \sum_{i=1}^k\int_\R \int_{S^1}(w(R_\th\s)-w(\th)) \psi_i(R_\th\s) D^{k-1}_\phi\cK(\phi,\th,R_\th\s,\rho)[\psi_j]_{ \stackrel{j=1,\dots,k}{j\neq i}} d\s d\rho.
\end{align*}
Then by  Lemma \ref{lem:est-cand-deriv}, we have 
\begin{align}\label{eq:est-cF1}
\|\ov{h}  \|_{C^{\b-\e}(S^1)} & \leq  c\left(1+ \|{\phi }\|_{C^{1,\b}(S^1)} \right)^{c}   \prod_{i=1}^k \|{\psi}_i\|_{C^{1, \b}} \|w\|_{C^{\b}(S^1)}      .
\end{align}
We can now follow step by step the arguments in \cite{Cabre2015B}   to deduce that $h:\cA_\d\subset C^{1,\b}(S^1)\to C^{0,\b-\e}(S^1)$  is of class $C^\infty$. Moreover, by \eqref{eq:Dk-LT2}, we get  
$$
D^k h(\phi)[\psi_1,\dots,\psi_k](\th)=\ov{h}(\th). 
$$
Now, we pick  a $w\in C^{1,\b}(\R)$, with $\|w \|_{C^{1,\b}(\R)}<\d/2$. Then for   $\phi\in  \cA_{\d/2}$,    we have
\begin{align*}
[\cF_{0,N}(\phi+ w)-\cF_{0,N}(w)- h(\phi )](\th)&= \int_0^1 [h(\phi+ w) -  h(\phi )](\th)\, d\varrho\\
&= \int_0^1 \varrho\int_0^1 D h(\phi+ \rho \varrho w)[w]  (\th)\,d\rho d\varrho,
\end{align*} 
recall the expression of $h$ in \eqref{eq:def-h-deriv-cF0N}.
  Since $\phi+\rho \varrho w\in \cA_{\d}$ for $\rho,\varrho\in (0,1)$, it follows   from Lemma \ref{Propo-reg-h} that  
\begin{align*}
\|\cF_{0,N}(\phi+ w)-\cF_{0,N}(\phi)- &h(\phi )\|_{C^{0,\a}(S^{N-1})}\leq  \int_0^1 \int_0^1 \|h(\phi+ \rho \varrho  w)[w]\|_{C^{0,\b}(S^{N-1})}\,  d \rho d\varrho\\
&\leq c  \, \|w\|_{C^{0,\b}(\R)}^2(1+\|\phi\|_{C^{1,\b}(S^{N-1})}+\|w\|_{C^{0,\b}(S^{N-1})} )^c.
\end{align*} 
From this, we conclude that $\cF_{0,N}$ is differentiable on $\cA_{\d/2}$ with 
$$
\cF_{0,N}(\vp)[w]=h(\vp).
$$
The $C^\infty$-character of $\cF_{0,N}$ on $\cA$ now follows from the one of $h$ and the fact that $\d$ is an arbitrary positive number.
 \QED
\subsubsection{Completion of the proof of Proposition \ref{Propo-reg-h} for $N=3$.}
As we will see, this case is much easier because
\be \label{eq:int-s2-finite}
 \int_{S^2}|\s-e|^{-1}d\s<\infty.
 \ee
Here, $h$ takes the form (recall \eqref{eq:int-cos-G-k-sqrt}) 
     \begin{align}
h(\phi)(\th) & =  \int_{S^{2}}\frac{w(R_\th\s)-w(\th) }{ |\s-e|}\cK_3(\phi,\th,R_\th\s )\phi^2 (R_\th\s )  \, d\s , \nonumber \\
 \end{align}
where $\cK_3:\cA_\d\times S^{ 2}\times S^{2} \to \R$ is given by 
$$
\cK_3(\phi,\th,\s):=  G_{|\th-\s|\k}\left(\L_0(\phi, \th,\s )^2+\phi(\s)\phi(\th) \right),
$$
with, as above, 
$$
\L_0(\phi, \th,\s)=\frac{    \phi(\th)-\phi( \s) }{|\th-\s|}.
$$
  Moreover by  \eqref{eq:est-deriv-G-a-latt}, \eqref{eq:estL2-th-sig} and  \eqref{eq:estL2-th1-sig1}, we have the estimates 

$$
\|D^k_\phi\cK_3(\phi,\th,R_\th\s ) \|\leq c (1+\|\phi\|_{C^{1,\b}(S^2)})^c  
$$
and
\begin{align*}
 \|D^k_\phi\cK_3(\phi,\th_1,R_{\th_1}\s  )&- D^k_\phi \cK(\phi,\th_2,R_{\th_2}\s )\leq c (1+\|\phi\|_{C^{1,\b}(S^2)})^c |\th_1-\th_2|^\b,
\end{align*}
which holds 
for every $e\in S^2, k\in \N,\d>0$,  $\phi\in \cA_\d$, $\rho\in \R$, $\th, \th_1,\th_2\in S_e$ and $ \s\in S^2$.
We define 
\begin{align*}
\hat{h}(\th)&=\sum_{i=1}^k  \int_{S^2}\frac{w(R_\th\s)-w(\th)}{|\s-e|} \phi(R_\th\s) D^k_\phi \cK_3(\phi,\th,R_\th\s)[\psi_1,\dots,\psi_k]d\s\\
&+ \sum_{i=1}^k  \int_{S^2}\frac{w(R_\th\s)-w(\th)}{|\s-e|}  \psi_i(R_\th\s) D^{k-1}_\phi\cK_3(\phi,\th,R_\th\s)[\psi_j]_{ \stackrel{j=1,\dots,k}{j\neq i}} d\s.
\end{align*}
Then, from the estimates of the kernel $\cK_3$ above and \eqref{eq:int-s2-finite}, we can deduce that 
\begin{align}\label{eq:est-cF000}
\|\hat{h}  \|_{C^{\b}(S^2)} & \leq  c\left(1+ \|{\phi }\|_{C^{1,\b}(S^2)} \right)^{c}   \prod_{i=1}^k \|{\psi}_i\|_{C^{1, \b}} \|w\|_{C^{\b}(S^2)} .     
\end{align}

 Now similar arguments as in the case $N=2$, show that $h:\cA_\d\to C^{0,\b}(S^2)$ is of class $C^\infty$ and $D^k h[\psi_1,\dots,\psi_k](\phi)(\th)= \hat h(\th)$. The proof of Proposition \ref{Propo-reg-h}  is then completed.
 \QED

\end{document}